\documentclass[10pt]{amsart}
\usepackage{amsmath}
\usepackage{amsthm}
\usepackage{amssymb}
\usepackage{graphics}
\usepackage{bm}
\usepackage{setspace}

\usepackage{amsfonts}
\usepackage{latexsym}
\usepackage{mathdots}
\usepackage{mathrsfs}
\usepackage{enumitem}
\usepackage{eucal}

\usepackage{color}

\newcommand{\be}{\begin{equation}}
\newcommand{\ee}{\end{equation}}
\newcommand{\ba}{\begin{eqnarray}}
\newcommand{\ea}{\end{eqnarray}}
\newcommand{\bi}{\begin{itemize}}
\newcommand{\ei}{\end{itemize}}
\newcommand{\bn}{\begin{enumerate}}
\newcommand{\en}{\end{enumerate}}
\newcommand{\bbm}{\begin{bmatrix}}
\newcommand{\ebm}{\end{bmatrix}}
\newcommand{\bp}{\begin{proof}}
\newcommand{\ep}{\end{proof}}
\newcommand{\nn}{\nonumber}

\newcommand{\mr}{\ensuremath{\mathrm}}

\newcommand{\mbf}{\ensuremath{\mathbf}}
\newcommand{\mc}{\ensuremath{\mathcal}}

\newcommand{\ov}{\ensuremath{\overline}}

\newcommand{\wt}{\ensuremath{\widetilde}}

\newcommand{\Om}{\ensuremath{\Omega}}

\newcommand{\la}{\ensuremath{\lambda }}

\def\C{\mathbb{C}}

\def\D{\mathbb{D}}
\def\T{\mathbb{T}}

\def\N{\mathbb{N}}
\def\B{\mathbb{B}}

\newcommand{\n}{\ensuremath{\mathbf{n} }}

\renewcommand{\H}{\ensuremath{\mathcal{H} }}
\newcommand{\J}{\ensuremath{\mathcal{J} }}

\newcommand{\K}{\ensuremath{\mathcal{K} }}
\renewcommand{\L}{\ensuremath{\mathscr{L} }}
\newcommand{\F}{\ensuremath{\mathbb{F} }}

\newcommand{\ip}[2]{\ensuremath{\langle {#1} , {#2} \rangle}}

\newcommand{\dom}[1]{\ensuremath{\mathrm{Dom} ({#1}) }}
\renewcommand{\dim}[1]{\ensuremath{\mathrm{dim} \left( {#1} \right) }}
\newcommand{\ran}[1]{\ensuremath{\mathrm{Ran} \left( {#1} \right) }}
\renewcommand{\ker}[1]{\ensuremath{\mathrm{Ker} ({#1}) }}

\newcommand{\im}[1]{\ensuremath{\mathrm{Im} \left( {#1} \right) }}
\newcommand{\re}[1]{\ensuremath{\mathrm{Re} \left( {#1} \right) }}

\newcommand{\Ext}[1]{\ensuremath{\mathrm{Ext} ({#1}) }}

\numberwithin{equation}{section}
\numberwithin{subsection}{section}

\newtheorem{thm}[subsection]{Theorem}

\newtheorem{lemma}[subsection]{Lemma}
\newtheorem{prop}[subsection]{Proposition}
\newtheorem{cor}[subsection]{Corollary}

\theoremstyle{definition}
\newtheorem{defn}[subsection]{Definition}
\newtheorem{remark}[subsection]{Remark}
\newtheorem{eg}[subsection]{Example}

-.5truein

\advance\voffset by -.1truein 

\expandafter\def\expandafter\normalsize\expandafter{%
    \normalsize
    \setlength\abovedisplayskip{10pt}
    \setlength\belowdisplayskip{10pt}
    \setlength\abovedisplayshortskip{7pt}
    \setlength\belowdisplayshortskip{7pt}
}

\title[AC theory for Drury-Arveson space]{Aleksandrov-Clark theory for Drury-Arveson space}

\author{M.T. Jury}
\address{University of Florida}
\email{mjury@ad.ufl.edu}

\author{R.T.W. Martin}
\address{University of Cape Town}
\email{rtwmartin@gmail.com}

\thanks{Second author acknowledges support of NRF CPRR Grant 90551.}

\begin{document}
\bibliographystyle{unsrt}
\begin{abstract}

    Recent work has demonstrated that Clark's theory of unitary perturbations of the backward shift restricted to a deBranges-Rovnyak subspace of Hardy space on the disk has a natural extension to the several variable setting. In the several variable case, the appropriate generalization of the Schur class of contractive analytic functions is the closed unit ball of the Drury-Arveson multiplier algebra and the Aleksandrov-Clark measures are necessarily promoted to positive linear functionals on a symmetrized subsystem of the Cuntz-Toeplitz operator system $\mc{A} + \mc{A} ^*$, where $\mc{A}$ is the non-commutative disk algebra.

We continue this program for vector-valued Drury-Arveson space by establishing the existence of a canonical `tight' extension of any Aleksandrov-Clark map to the full Cuntz-Toeplitz operator system.  We apply this tight extension to generalize several earlier results and we characterize all extensions of the Aleksandrov-Clark maps.

\end{abstract}

\maketitle
\onehalfspace

\vspace{5mm}   \noindent {\it Key words and phrases}:
Hardy space, Drury-Arveson space, model subspaces, deBranges-Rovnyak spaces, multiplier algebra, non-commutative disk algebra, Aleksandrov-Clark measures

\vspace{3mm}
\noindent {\it 2010 Mathematics Subject Classification} ---46L07 Operator spaces and completely bounded maps; 47B32 Operators on reproducing kernel Hilbert spaces; 46E22 Hilbert spaces with reproducing kernels; 47L80 Algebras of specific types of operators; 46J15 Banach algebras of differentiable or analytic functions

\section{Introduction}

    The Drury-Arveson space, $H^2 _d$, of analytic functions on the open unit ball $\B ^d := (\C ^d ) _1$ of $d-$dimensional complex space is the reproducing kernel Hilbert space (RKHS) $\H (k)$ of functions on $\B ^d$
corresponding to the positive kernel function $k : \B ^d \times \B ^d \rightarrow \C$:
$$ k (z,w) := \frac{1}{1-zw^*}; \quad \quad z,w \in \B ^d. $$ Here $zw^* := (w ,z) _{\C ^d} = z_1 \ov{w} _1 + ... + z_d \ov{w} _d, $ all inner products are assumed conjugate linear in the first argument.  In the case where $d=1$ we recover the classical Hardy space $H^2 = H^2 (\D )$ of analytic functions on the disk $\D$ which have non-tangential boundary values almost everywhere on the unit circle $\T$ with respect to normalized Lebesgue measure $m$.

Any reproducing kernel Hilbert space $\H (k)$ on a set $X$ is naturally equipped with a \emph{multiplier algebra}, $\mr{Mult} (\H (k) )$, the algebra of all functions on $X$ which multiply elements of $\H (k)$ into $\H (k)$:
$$ \mr{Mult} (\H (k) ) := \{ F : X \rightarrow \C | \ F h \in \H (k) \ \forall h \in \H (k) \}. $$ Identifying $\mr{Mult} (\H (k) ) $ as linear transformations on $\H (k)$, a closed graph theorem argument shows that $\mr{Mult} (\H (k)) \subset \L ( \H (k) )$ consists of bounded linear maps.  It is also straightforward to check that the multiplier algebra is closed in the weak operator topology (WOT).

The Schur-class for Drury-Arveson space is the closed unit ball $[H^\infty _d ] _1$ of the multiplier algebra. Again in the case where $d=1$, we recover the usual Banach algebra $H^\infty = H^\infty (\D )$ of bounded analytic functions on $\D$ and the Schur class of contractive analytic functions on the disk.

Given any Schur class $b \in [H^\infty _d ] _1$,
$$ k^b (z,w) := \frac{1 - b(z) b(w)^*}{1-zw^*}, $$ defines a positive kernel function on $\B ^d$ and one defines the deBranges-Rovnyak space $K(b) := \H (k^b )$ to be the corresponding reproducing kernel Hilbert space (RKHS) of analytic functions on $\B ^d$.  One can check that $k-k^b$ is again a positive kernel function and standard RKHS theory then implies that $K(b)$ is contained contractively in $H^2 _d$ \cite[Corollary 5.3]{Paulsen-rkhs}.
Recall that multiplication by the independent variable $z$ defines an isometry $S$ on $H^2 = H^2 (\D )$ called the \emph{shift}. The shift plays a central role in the theory of Hardy spaces \cite{Hoff}.  The adjoint $S^*$ of the shift is called the backward shift,
$$ (S^* f) (z) = \frac{f(z) - f(0) }{z}; \quad \quad f \in H^2, \ z \in \D. $$ Every deBranges-Rovnyak subspace of $H^2$ is invariant under $S^*$, and the restrictions of the backward shift to deBranges-Rovnyak spaces can be used to construct a functional model (a special case of the deBranges-Rovnyak functional model) for arbitrary completely non-coisometric (c.n.c) contractions \cite{dBss,dBmodel,Nik1986,Ball1987,Ball2011dBR,Sarason-dB}. (The full deBranges-Rovnyak model is constructed using a two-component RKHS with $K(b)$ as its $(1,1)$ entry).)  This is one important reason for interest in these spaces from the point of view of operator theory.  The theory we develop in this paper can be applied to extend this deBranges-Rovnyak model to a class of generally non-commuting row contractions which generalize c.n.c. contractions with equal defect indices \cite{Ball2011dBR,Martin-dBmodel}. Standard references for the theory of deBranges-Rovnyak spaces on the disk are \cite{Sarason-dB,dBss}, and most of the basic theory we will generalize can be found in \cite{Sarason-dB}.

We are motivated by the theory developed by D.N. Clark in \cite{Clark1972} concerning unitary perturbations of the restriction of the backward shift to a deBranges-Rovnyak space $K (b)$.  Clark considered the case of inner $b$ in which case $K(b)$ is a co-invariant model subspace of $H^2$. We will closely follow the extension of this theory to the Schur class of contractive analytic functions on the disk as presented in \cite{Sarason-dB}, and as extended to Drury-Arveson space in \cite{Jur2014AC}.

\subsection{Clark Theory in the classical ($d=1$) case}

There is a natural bijection between the Schur class $[H^\infty ] _1$ of (purely) contractive analytic functions on the unit disk and the Herglotz class of all analytic functions on the disk with non-negative real part given by
$$ b \mapsto H_b := \frac{1+b}{1-b} ; \quad \quad \mbox{and} \quad \quad H \mapsto b_H := \frac{H-1}{H+1}. $$ (Any Schur class $b$ is either purely contractive, \emph{i.e.} $|b (z) | <1$, $\forall \ z \in \D$ or $b$ is a unimodular constant.)
There is also a natural bijection between Herglotz functions modulo imaginary constants and non-negative Borel measures on the unit circle given
by the Herglotz representation formula: given any Herglotz function $H$ on the disk, there is a unique non-negative Borel measure $\mu $ on the unit circle so that
 \be H(z) = i \im{H (0)} + \int _\T \frac{1+z\zeta^*}{1 -z\zeta ^*} \mu (d\zeta ). \label{classHerg} \ee In the above $\zeta ^* := \ov{\zeta}$ denotes complex conjugate. Conversely, given any non-negative Borel measure on $\T$ (and any imaginary constant), this formula defines a Herglotz function on the disk.
It follows that one can associate a unique non-negative Borel measure $\mu _b$ on the unit circle to any Schur class $b$. We will refer to this measure as the \emph{Herglotz measure} of $b$. More generally there is a $\mc{U} (1)$-parameter family (the one-dimensional unitary group, identified with the unit circle $\T$) of measures naturally associated with $b$, the \emph{Aleksandrov-Clark measures}. Namely, given any contractive analytic function $b$ and any $\alpha \in \T$, the Aleksandrov-Clark (AC) measure $\mu _\alpha$ is defined to be $\mu _{b \alpha ^*}$, the Herglotz measure of the contractive analytic function $b \alpha ^*$,
$$ \frac{1 +b(z) \alpha ^*}{1-b(z) \alpha ^*} = i \im{\frac{1 + b(0) \alpha ^*}{1-b(0) \alpha ^*}} + \int _\T \frac{1+z \zeta^*}{1-z\zeta ^*} \mu _\alpha (d\zeta). $$

 For any non-negative Borel measure $\mu$ on the circle $\T$, let $L^2 (\mu )$ denote the Hilbert space of $\mu -$square integrable functions on $\T$.  Since $\mu = \mu _b$ for a unique Schur class $b \in [H^\infty ] _1$, we will often use the notation $L^2 (b)$ for $L^2 (\mu)$. Let $P^2 (b)$ denote the closure of the analytic polynomials in $L^2 (b)$, \emph{i.e.}
$$ P^2 (b) := \bigvee _{n \geq 0} \zeta ^n, $$ and let $P^2 _0 (b)$ be the closed linear span of the non-constant monomials in $L^2 (b)$,
$$ P^2 _0 (b) := \bigvee _{n \geq 1} \zeta ^n \subset P^2 (b). $$

The Schur class is a convex set and $b \in [H^\infty ] _1$ is an extreme point if and only if $1- |b|$ fails to be log-integrable:
$$ \int _\T \ln ( 1 - |b(\zeta )| ) = - \infty \quad \Leftrightarrow \ b \ \mbox{is extreme}, $$ \cite[Chapter 9]{Hoff}. Since the Radon-Nikodym derivative of any AC measure $\mu _\alpha$ for $b$ with respect to normalized Lebesgue measure $m$ is \cite[Proposition 9.1.14]{Ross2006CT}:
$$ \frac{d\mu _\alpha}{dm} (\zeta ) = \frac{ 1 - |b (\zeta ) | ^2 }{|1 - b(\zeta) \alpha ^* | ^2}, $$ it follows that $b$ is an extreme point if and only if
$$ \int _\T \ln \left( \frac{d\mu _\alpha}{dm}  \right) dm = - \infty, $$ so that Szeg\"{o}'s theorem implies that $b$ is an extreme point if and only if $P^2 (b) = P^2 _0 (b)$, \emph{i.e.} if and only if the closed linear span of the non-constant analytic monomials contains all of the analytic polynomials in $L^2 (b)$ \cite[Chapter 4]{Hoff}.  It is not further not hard to show that $P^2 (b) = P^2 _0 (b)$ if and only if $P^2 (b) = L^2 (b)$.

In the seminal paper \cite{Clark1972}, D.N. Clark established the following results for the case of inner $b$ (the general versions below can be found in \cite[Chapter III]{Sarason-dB}).

\begin{thm} \label{classCT}
    For any $\alpha \in \T$ and any $b \in [H^\infty ] _1$ the \emph{weighted Cauchy} or \emph{Fantappi\`{e} transform} $\mc{F} _\alpha $ defined by
    $$ (\mc{F} _\alpha f ) (z) := (1 - b(z) \alpha ^* ) \int _\T \frac{f(\zeta )}{1-z \zeta ^* } \mu _\alpha ( d\zeta ), $$ is a unitary transformation from $P^2 (\mu _\alpha ) = P^2 (b \alpha ^* )$ onto
the deBranges-Rovnyak space $K (b)$.
\end{thm}

Given any Schur class $b$ and $\alpha \in \T$, let $Z ^\alpha := Z ^{b\alpha ^*}$ denote the unitary operator of multiplication by the independent variable in $L^2 (\mu _\alpha ) = L^2 (b \alpha ^* )$.
Clearly $P^2 (b \alpha ^* )$ is an invariant subspace for $Z ^\alpha $. Let $Y ^\alpha := Z^\alpha | _{P^2 (\mu _\alpha )}$, an isometry which equals $Z^\alpha $ if and only if $b$ is an extreme point.
For simplicity assume $b(0) = 0$ and let $X := S^* | _{K(b)}$.

\begin{thm} \label{uniperturb}
     Given $b \in [H^\infty ] _1$ (assume $b(0) = 0$), the weighted Cauchy transform $\mc{F} _\alpha$ intertwines the co-isometry $(Y ^\alpha ) ^*$ with a rank-one perturbation of $X$:
$$ X ^\alpha := \mc{F} _\alpha (Y ^\alpha ) ^* \mc{F} _\alpha ^* = X + \ip{\cdot}{1} S^* b \alpha ^*. $$
The point evaluation vector at $0$, $k_0 ^b \equiv 1 \in K (b)$ is cyclic for each $X ^\alpha$.

If $b$ is an extreme point of the Schur class then $Y ^\alpha = Z^\alpha $ is unitary so that each $X ^\alpha$ is a rank-one unitary perturbation of the restricted backward shift $X$. In this case if
$P _\alpha$ denotes the projection-valued measure of $X ^\alpha$ then $\mu _\alpha ( \Om ) = \ip{P _\alpha (\Om ) 1}{ 1 }.$
\end{thm}

\begin{remark}
    In the case where $b$ is an extreme point, the inverse of the weighted Cauchy transform $\mc{F} _\alpha$ implements a spectral realization for the unitary operator $X ^\alpha$.
\end{remark}

\subsection{The several variable case}

These and other related results were recently generalized to the several variable case of Drury-Arveson space $H^2 _d$ by the first author \cite{Jur2014AC}. The several variable generalizations of Clark's results as presented in \cite{Sarason-dB} demonstrate that many of the proofs are algebraic and dimension-free.  Fascinatingly, as soon as $d>1$, this theory applied to the commutative operator algebra $H^\infty _d$ reveals fundamental connections to non-commutative operator algebra theory, namely to Popescu's noncommutative Hardy space theory, and free semigroup algebra theory \cite{Jur2014AC,Pop96disk,Davidson2001}.

In contrast with the $d=1$ case, it is well known that not every bounded analytic function on the ball $\B ^d$ is an element of $H^\infty _d$. It is not difficult to show that a contractive analytic function $b$ on the unit
ball $\B ^d$ belongs to the Schur class $[H^\infty _d ] _1$ if and only if the deBranges-Rovnyak kernel:
$$ k^b (z,w) := \frac{1 - b(z) b(w) ^*}{1-zw^*}, $$ defines a positive kernel function on $\B ^d$ \cite{Sha2013}. Similarly, not every analytic function $H$ with non-negative real part on the ball can be realized
as $H_b := (1-b) ^{-1} (1+b)$ for a Schur class $b \in [ H^\infty _d ] _1$ \cite{McPutinar2005,Jur2010Herglotz}. If $H = H_b$ for some Schur class $b \in [H^\infty _d ]_1$ we say that $H$ belongs to the \emph{Herglotz-Schur} class.  Perhaps even more remarkably, and again in contrast with the $d=1$ case, the direct analogue of the classical Herglotz representation formula does not hold in the several variable setting: not every $H$ in the Herglotz-Schur class can be realized as the integral of the Herglotz integral kernel $\frac{1+z\zeta ^*}{1-z\zeta ^*}$ with respect to a non-negative Borel measure $\mu$ on the boundary of the ball $\partial \B ^d$.  Instead,
as shown in \cite{McPutinar2005,Jur2010Herglotz,Jur2014AC}, one needs to replace the Herglotz integral kernel with a `non-commutative kernel' which takes values in a certain operator subsystem of the Cuntz-Toeplitz operator system, and the measure $\mu $ with a positive linear functional on this operator system \cite{McPutinar2005,Jur2014AC}.

Beginning with this observation, natural analogues of the above theorems of Clark and related results were obtained in the several variable case of $b \in [H^\infty _d ]_1$ \cite{Jur2014AC}. To obtain a suitable generalization of the unitary perturbation theorem it was assumed that the Schur class function $b$ was \emph{quasi-extreme}, a property generalizing the Szeg\"{o} approximation property: $L^2 (b) = P^2 (b) = P^2 _0 (b)$ from the classical single variable case. An open problem that made further generalization difficult in the several variable setting was whether or not any AC functional $\mu _b$ had a canonical \emph{tight} extension to the full Cuntz-Toeplitz operator system \cite[Question 3.6]{Jur2014AC}.

Instead of providing a full summary of the results of \cite{Jur2014AC}, we will proceed with developing the theory for the multiplier algebra of vector-valued Drury-Arveson space $H^2 _d \otimes \H$.  In this setting the
AC maps are promoted to completely positive maps into $\L (\H)$. This does not significantly complicate the analysis from the scalar-valued case, for the most part.

\subsection{Outline}

In the following section, Section \ref{Herglotz}, we develop the noncommutative Cauchy or Fantappi\`{e} transform and Herglotz representation formulas for the vector-valued case. Our approach is slightly modified from that of \cite{McPutinar2005,Jur2014AC} and makes use of a partial $d$-isometry $V ^b$ acting on a reproducing kernel Hilbert space $\L (b)$ which we call the \emph{Herglotz space} associated to $b \in [H^\infty _d \otimes \L (\H )]_1$.

In Section \ref{extensions} we apply our Herglotz space framework and the partial isometry $V^b$ to construct a natural completely positive (CP) extension $\nu _b$ of the Aleksandrov-Clark CP map $\mu _b$ to the full Cuntz-Toeplitz operator system, and we prove that this extension is the unique \emph{tight} extension in the sense of \cite[Definition 3.2]{Jur2014AC}.  We then show that the set of all extensions of $\mu _b$ can be parametrized by cyclic isometric extensions of this partial isometry $V^b$ and several equivalent characterizations of the quasi-extreme Szeg\"{o} approximation property are developed.

Section \ref{Gleasonsection} contains our results on the Gleason problem for $K(b)$ and our generalization of Clark's unitary perturbation results \cite{Sarason-dB}. Solutions to the Gleason problem are the appropriate several variable analogue of the restriction of the backward shift to a deBranges-Rovnyak space in the single variable case. We show that the set of all contractive Gleason solutions for $K(b)$ is parametrized by the set of all contractive extensions of the partial $d-$isometry $V^b$ on the Herglotz space $\L (b)$. The equivalent characterizations of quasi-extremity are summarized in Theorem \ref{summarythm}.

Finally, Section \ref{Examplessection} gives some examples of the foregoing constructions in the case of inner $b$.

\subsection{Vector-valued RKHS}

    We will be working with vector-valued reproducing kernel Hilbert spaces (RKHS) of analytic functions on the unit ball $\B ^d = (\C ^d ) _1$.  Recall the following basic facts from RKHS theory:
    
Given a set $X \subset \C ^d$, and an auxiliary Hilbert space $\H$, a vector-valued RKHS $\K$ on $X$ is a Hilbert space of $\H$-valued functions on $X$ so that for any $x \in X$ the linear point evaluation maps
$K_x ^* \in \L ( \K , \H )$ defined by $$ K_x ^* F = F(x) \in \H; \quad \quad  F, \in \K $$ are bounded. We write $K_x := (K _x ^* ) ^* \in \L (\H , \K )$ for the Hilbert space adjoint. The
operator-valued function $ K : X \times X \rightarrow \L (\H )$:
$$ K (x,y) :=  K_x ^* K_y \in \L (\H ); \quad \quad x,y \in X, $$ is called the \emph{reproducing kernel} of $\K$. One usually writes $\K = \H (K )$. The reproducing kernel $K$ of any vector-valued RKHS on $X$ is
a \emph{positive kernel function} on $X$: A function $K : X \times X \rightarrow \L (\H)$ is an operator-valued positive kernel function on $X$ if for any finite set $\{ x_k \} _{k=1} ^N \subset X$, the matrix
$$ [ K (x_i , x_j ) ]  \in \L (\H ) \otimes \C ^{N\times N}, $$ is non-negative. The (vector-valued extension of the) theory of RKHS developed by Aronszajn and Moore (see \emph{e.g.} \cite{Paulsen-rkhs}) shows that there is a bijection between positive $\L (\H )$-valued kernel functions on $X\times X$ and RKHS of $\H$-valued functions on $X$. Namely, given any positive kernel $K$ on $X$ there is a RKHS $\K$ on $X$ so that $K$ is its reproducing kernel, $\K = \H (K)$. If $F:X\to \H$ is a function in $\K$, then the kernel $K$ reproduces the value $F(x)\in\H$ at the point $x\in X$ in the sense that for all $h\in\H$,
$$ \langle F(x), h\rangle_{\H} = \langle F, K_x h\rangle_{\H(K)}.$$

\subsection{Drury-Arveson, deBranges-Rovnyak and Herglotz spaces}

This paper takes place in the setting of vector-valued Drury-Arveson space $H^2 _d \otimes \H$, where $\H$ is finite dimensional or separable. This is the vector-valued reproducing kernel Hilbert space $\H (k)$ of $\H$-valued functions on the ball $X = \B ^d = (\C ^d ) _1$ corresponding to the several variable operator-valued Szeg\"{o} kernel: $$ k (z,w) := \frac{1}{1-zw^*} I _\H. $$
We will use the notation $H^\infty _d \otimes \L (\H ) := \mr{Mult} (H^2 _d \otimes \H )$ (the multiplier
algebra is the closure of this algebraic tensor product in the weak operator topology on $H^2 _d \otimes \H$). The \emph{Schur class} is the closed unit ball of this multiplier algebra, $[H^\infty _d \otimes \L (\H )] _1$.

Recall that not every contractive analytic $\L (\H )$-valued function $b$ belongs to $[H^\infty _d \otimes \L (\H)] _1$. Given such a $b$, it is not hard to check that $b$ belongs to the Schur class of vector-valued Drury-Arveson space if and only if $$ k^b (z, w ) := \frac{I - b(z) b(w) ^* }{1-zw^*} \in \L (\H); \quad \quad z,w \in \B ^d $$ defines a positive $\L (\H )$-valued kernel function on $\B ^d \times \B ^d$ \cite{BT1998DFP, Sha2013}. The deBranges-Rovnyak space $K(b) := \H (k^b)$ is defined as the corresponding RKHS of $\H$-valued analytic functions on $\B ^d$.  As in the classical
case it is straightforward to verify that $k - k^b$ (where $k$ is the Szeg\"{o} kernel for $H^2 _d \otimes \H$) is again a positive $\L (\H )$-valued kernel function on $\B ^d$ so that vector-valued RKHS theory implies that $K(b)$ is contained contractively in $H^2 _d \otimes \H$ \cite[Theorem 10.20]{Paulsen-rkhs}. That is, $K(b) \subset H^2 _d \otimes \H$ as vector spaces, and the injection is a contraction.

The \emph{Herglotz-Schur class} is the set of all $\L (\H )$-valued functions $H$ on $\B ^d$ with non-negative real part so that
$$ K (z,w) := \frac{1}{2} \frac{H(z) + H(w) ^*}{1-zw^*} \in \L (\H ), $$ defines a positive kernel function on $\B ^d$. A similar argument to \cite[Proposition 2.1, Chapter V]{NF} shows that any $b \in [H^\infty _d \otimes \L (\H ) ] _1$ decomposes as $b = b_0 + b_1$ on $\H = \H _0 \oplus \H _1$ where $\| b_0 (z) \| < 1$, $\forall z \in \B ^d$ is purely contractive, or \emph{pure}, and $b_1$ is a constant isometry on $\B ^d$ from $\H _1$ onto its range in $\H$. We assume throughout that $b=b_0$ is purely contractive so that $I -b(z)$ is invertible for $z \in \B ^d$. As before, there is a bijection between purely contractive Schur class functions $b \in [H^\infty _d \otimes \L (\H ) ] _1$ and pure Herglotz-Schur functions $H$ (Herglotz-Schur functions for which $H(z) +I$ is invertible on $\B ^d$) given by:
$$ H (z) = H_b (z):= (I-b (z) ) ^{-1} (I+b (z)); \quad \quad b \in [H^\infty _d \otimes \L (\H ) ]_1, $$ and
$$ b (z) = b_H (z) := (H(z) +I) ^{-1} (H (z) -I). $$
If $H = H_b$ is a pure Herglotz-Schur function, then the kernel $K^b$ can be expressed as
\be K^b (z,w) =\frac{1}{2} \frac{H_b (z) + H_b (w) ^* } {1-zw^*} = (I -b(z) ) ^{-1} k^b (z,w) (I- b(w) ^* ) ^{-1}. \label{Herglotzkernel} \ee In this case where $H = H_b$, for a purely contractive Schur class $b \in [H^\infty _d \otimes \L (\H) ] _1$, we define the \emph{Herglotz space} of $b$ to be $\L (b) := \H (K ^b)$, the RKHS  of $\H$-valued functions on $\B ^d$ with reproducing kernel $K ^b$. The above relationship between the kernels $k^b$ of $K(b)$ and $K^b$ of $\L (b)$ (for purely contractive $b$) implies that there is an isometric multiplier $U_b : K(b) \rightarrow \L (b)$:
\begin{lemma} \label{ontoisomult}
The map $U_b : K(b) \rightarrow \L (b)$ defined by multiplication by
\be  U_b (z) := (I-b(z) ) ^{-1}, \ee is an onto isometry. The action of $U_b$ on point evaluation kernels is
$$ U_b k_z ^b = K_z ^b (I -b(z) ^* ). $$
\end{lemma}
We will omit the superscript and subscript $b$ when this is clear from context.
\section{Herglotz representation formula and Fantappi\`{e} transform}
\label{Herglotz}

A bit of straightforward algebra using the formula (\ref{Herglotzkernel}) above for the Herglotz reproducing kernel $K = K^b$, shows
\ba  (K_z - K_0 )^* (K_w - K_0) & = & zw^* K_z ^* K_w \nonumber \\
& = &  (z^* K_z )^* (w^* K_w); \quad \quad z,w \in \B ^d. \nonumber \ea
In the above, $z$ is viewed as a strict contraction from $\L (b) \otimes \C ^d$ into $\L (b)$ so that $z^* K_z \in \L (\H , \L (b) \otimes \C ^d)$ obeys
$$ z^* K_z h := \begin{bmatrix} \ov{z_1} K_z h \\ \vdots \\ \ov{z_d} K_z h \end{bmatrix} = \begin{bmatrix} \ov{z_1} \\ \vdots \\ \ov{z_d} \end{bmatrix} K_z h \in \L (b) \otimes \C ^d. $$

It follows that one can define a partial $d-$isometry on $\L (b)$ as follows: Set $$\dom{\check{V}} := \bigvee _{w\in \B^d; \ h \in \H} w^* K_w h \subset \L (b) \otimes \C ^d, $$ and
$$ \ran{\check{V}} := \bigvee _{w \in \B^d; \ h \in \H} (K_w - K _0 ) h \subset \L (b).$$ Here and throughout $\bigvee$ denotes closed linear span.
The above calculations show that the linear map $\check{V} : \dom{\check{V}} \rightarrow \ran{\check{V}}$ defined by
$$ w^* K_w h \mapsto (K_w - K_0 ) h, $$ is an isometry from its domain, $\dom{\check{V}}$ onto its range, $\ran{\check{V}}$. Let $V = V^b$ be the partial isometric
extension of $\check{V}$ to all of $\L (b) \otimes \C ^d$ (which is zero on the orthogonal complement of $\dom{\check{V}}$ in $\L (b) \otimes \C ^d$). Then $V$ is a partial $d$-isometry on the $\H$-valued RKHS $\L (b)$.

It will be helpful to describe the orthogonal complements of the initial and final spaces of $V$: First observe that $F\in \L (b)$ is orthogonal to $\ran{V} = \cap _{z \in \B ^d} \ker{ (K_z -K_0 ) ^*}$ if and only if for all $w\in \B^d$,
  \begin{equation*}
    0 = (K_w ^* - K_0 ^*) F =F(w)-F(0),
  \end{equation*}
in other words,  $F(w)=F(0)$ for all $w$ and thus $F:\B^d\to \H$ is constant. Note that $V$ may be that $V$ is surjective in which case $\L(b)$ contains no non-zero constant functions.

Similarly, a $d$-tuple $\mbf{F} = (F_1, \dots F_d) ^T \in \L (b)\otimes \C^d$ (the superscript $T$ denotes transpose) belongs to $\ker{V} = \dom{\check{V}} ^\perp$ if and only if
$$  0= (w^* K_w ) ^* \mbf{F} = \sum _{j=1} ^d w_j F _j (w). $$
  When $d=1$ this condition can only hold if $F\equiv 0$, so that $\ker{V} ^\perp =\L (b)$. On the other hand when $d>1$ there can exist nontrivial solutions $(F_1, \dots F_d)$ to $\sum_{j=1}^d w_jF_j(w)=0$ in $\L (b)$. One can show that $\ker{V ^b}$ is never trivial when $d>1$, see Remark \ref{ntkernremark}.

\subsection{A CP map on a symmetrized Cuntz-Toeplitz operator subsystem}

Recall that the full Fock space $F^2 _d$ over $\C ^d$ is the direct sum of all powers of tensor products of $\C ^d$ with itself:
\ba  F^2 _d & := & \C \oplus \C ^d \oplus \left( \C ^d \otimes \C ^d \right) \oplus \left( \C ^d \otimes \C ^d \otimes \C ^d \right) \oplus ... \nn \\
& = & \bigoplus _{k=0} ^\infty \left( \C ^d \right) ^{k \cdot \otimes}. \nn \ea Given a fixed orthonormal basis $\{ e_k \}$ of $\C ^d$, the \emph{left creation operators} $L_k \in \L (F ^2 _d )$ are defined by tensoring on the left with $e_k$:
$$ L_k f := e_k \otimes f ; \quad \quad f \in F^2 _d. $$ Each $L_k$ is an isometry, the $L_k$ have orthogonal ranges ($L_k ^* L_j = \delta _{kj} I$) and
$L := (L _1 , ..., L_d ) : F^2 _d \otimes \C ^d \rightarrow F^2 _d$ defines a non-commuting row-isometry on $F^2 _d$ which we call the free or non-commutative shift. The non-commutative disk algebra $\mc{A} := \mc{A} _d$ is the unital
norm-closed algebra generated by the left creation operators,
$$ \mc{A} := \bigvee _{\alpha \in \F ^d} L ^\alpha. $$ Here $\F ^d$ denotes the unital free semigroup on $d$ letters (the unit is the empty word $\emptyset$ and one defines $L^\emptyset = I$). We call the corresponding operator system $(\mc{A} + \mc{A} ^*) ^{-\| \cdot \|}$ the \emph{Cuntz-Toeplitz operator system} (we will simply write $\mc{A} + \mc{A} ^*$ for this norm-closure).

By the Bunce-Frazho-Popescu dilation theorem \cite{Pop89iso}, $V=V^b$ has a minimal isometric dilation $W=W^b$ on $\K _b \supset \L (b)$ obeying $W^* | _{\L (b)} = V^*$. This shows that $W^*$ is a $d-$contractive extension
of $V^*$ in the sense that $W^*$ agrees with $V^*$ on the final space of $V$: $W^* (V V ^* ) = V^*$. We use the notation $D \supseteq V$ for any $d-$contractive extension of $V$ on $\J \supset \L (b)$. The following is a general fact that holds for contractive extensions of any partial isometry between Hilbert spaces:

\begin{lemma} \label{contractext}
If $D$ is a $d-$contraction on $\J \supset \L (b)$ then $V \subseteq D$ if and only if $V^* \subseteq D^*$.
\end{lemma}
\begin{proof}
Suppose that $V \subseteq D$. Since $D(V^*V) = V (V^* V)$ it follows that $(VV^*) D (V^* V) = D (V^* V)= V (V^*V) = (VV^*)V$. Taking adjoints shows $V^* (VV^*) = (V^*V) D^* (VV^*)$.
It follows that if $f = VV^*f$ is a unit norm element in $\ran{V}$ then
\ba 1 & \geq & \| D^* f \| ^2 = \| (V^*V) D ^* f \| ^2 + \| (I - V^* V ) D^* f \| ^2 \nonumber \\
& = & \| V^* f \| ^2 + \| (I- V^*V) D^* f \| ^2 \nonumber \\
& = & 1 + \| (I- V^*V) D^* f \| ^2. \nonumber \ea This proves that $D^* (VV^*) = (V^*V) D (VV^*) = V^* (VV^*)$
and $V^* \subseteq D^*$. The converse is similarly easy to prove.
\end{proof}

It follows that when $\J = \L(b)$, any $d-$contractive extension $D$ of $V$ has the form $V(Y) = V + Y$ where $Y : \ker{V} \rightarrow \ran{V} ^\perp$ is a contraction.

\begin{lemma} \label{RKext}
A $d$-contraction $D$ acting in a Hilbert space $\J\supset \L(b)$ is
an extension of $V$ if and only if
$$ K_z h = (I-z^*D) ^{-1} K_0 h.$$
\end{lemma}
   In particular this holds for $D=V$ or $D =W$, the minimal isometric dilation of $V$.
\begin{proof}
    Since the initial space of $V$ is spanned by vectors of the form $z^* K_z h $ and $D$ extends $V$, it follows that for any $z \in \B ^d$ and $h \in h$,
$$ D (z^* K_z h) = V (z^* K_z h) = (K_z - K_0)h. $$ Writing this out then shows:
$$ (D_1 \ov{z_1} + ... + D_d \ov{z_d}) K_z h = (K_z - K_0 ) h.$$ Solving for $K_0 h$ yields
$$ K_0 h = (I-Dz^*) K_z h, $$ and since $Dz^*$ is a strict contraction, one can invert this expression to obtain
$$ K_z h = (I-Dz^*) ^{-1} K_0 h.$$

On the other hand, if $K_z h = (I-Dz^*) ^{-1} K_0 h$ for all $z$ and
$h$, then the above steps reverse to show that $D (z^* K_z h) = (K_z - K_0)h = V (z^*
K_z h)$, and thus $D$ extends $V$.
\end{proof}

\begin{lemma} \label{starextend}
    Let $W$ be a $d-$isometry on a Hilbert space $\H$. The map $\pi _W : \mc{A} \rightarrow \L (\H )$ defined by $\pi _W (L^\alpha ) := W^\alpha$ is a completely isometric unital homomorphism
which obeys $\pi _W ( (L^\alpha ) ^* L ^\beta ) = (W^\alpha ) ^* W^\beta$, for all $\alpha, \beta \in \F ^d$. Moreover $\pi _W$ extends to a completely contractive unital $*-$homomorphism $\Pi _W : \mc{E} = C^* (\mc{A}) \rightarrow
\L (H)$ defined by $\Pi _W ( L ^\alpha (L^\beta ) ^* ) := \pi _W (L ^\alpha ) \pi _W (L^\beta ) ^*.$
\end{lemma}

    Any such map $\pi _W$ is the restriction of a $*$-representation $\Pi _W$ of $\mc{E}$, and is hence \emph{$*$-extendible} in the sense of \cite{DPac}.

\begin{proof}
    Since $W$ is a row-isometry it follows that $(W_k ) ^* W_j = \delta _{ij } I _\H$, and the relation $\pi ( (L^\alpha ) ^* L ^\beta ) = (W^\alpha ) ^* W^\beta$ follows from this.
The remaining assertions are standard results of Popescu \cite{Pop96disk,Pop98universal}
\end{proof}

Note here that by results of
\cite{Jur2014AC}, $$ \bigvee _{z \in \B^d} (I-z^*L )^{-1} = \bigvee _{\n \in \N ^d} L ^{\n} =: \mc{S} $$ is the norm-closed operator subspace of $\mc{A}$ spanned by the symmetrized monomials $L^{\n}$.
Here recall that $\N ^d$ is the unital additive semigroup of all $d$-tuples of non-negative integers, and if $\la : (\F ^d , \cdot) \rightarrow (\N ^d , + )$ is the letter counting map, $$ L^\n := \sum _{\la (\alpha ) = \n} L^\alpha. $$ For example, if $d=2$,
$$ L^{(1,2)} := L_1 L_2 ^2 + L_2 ^2 L_1 + L_2 L_1 L_2. $$ Also if $\n= (n_1 , ..., n_d ) \in \N ^d$ define $| \n | := n_1 + ... + n_d$. The symmetrized operator system $\mc{S} + \mc{S} ^*$,
as well as the full Cuntz-Toeplitz operator system $\mc{A} + \mc{A}^*$ enjoy the \emph{semi-Dirichlet property} \cite{Dav2011,Jur2014AC}:
$$ \mc{S} ^* \mc{S} \subset (\mc{S} + \mc{S} ^* ) ^{-\| \cdot \|} \quad \quad \mbox{and} \quad \quad \mc{A} ^* \mc{A} \subset (\mc{A} + \mc{A} ^*) ^{-\| \cdot \|}. $$ To simplify notation we will simply write $\mc{S} + \mc{S} ^*$ and $\mc{A} + \mc{A}^*$ in place of the norm-closed operator systems
$\ov{\mc{S} + \mc{S} ^* } ^{\| \cdot \|}$ and $\ov{\mc{A} + \mc{A} ^* } ^{\| \cdot \|}$. We will use the notations $CP (\mc{S} , \H )$ and $CP (\mc{A} , \H)$ for the sets of all completely positive maps from $\mc{S} + \mc{S} ^*$ and $\mc{A} + \mc{A} ^*$ into $\L (\H)$.

\begin{prop} \label{CPmap}
    Define $\mu _b : \mc{S} + \mc{S} ^* \rightarrow \L (\H )$ by
$$ \mu _b  \left( (I-zL^* ) ^{-1} (I - L w^* ) ^{-1} \right) := K^b (z,w). $$ Then $\mu _b$ is a completely positive map obeying
$$ \mu _b  \left( (I-zL^* ) ^{-1} (I - L w^* ) ^{-1} \right) = K_0 ^* (I-zV^*)^{-1} (I-Vw^* )^{-1} K_0, $$ and
$ \mu _b (L ^{\n} ) = K_0 ^* V^{\n} K_0. $
\end{prop}

\begin{proof}
    Given any $d-$contraction $T$ on a separable Hilbert space $\mc{H}$ the map $\phi _T (L ^\alpha ) := T ^\alpha$, $\alpha \in \F ^d$ defines a completely contractive unital map and so
extends to a CPU map on $\mc{A} + \mc{A} ^*$ \cite[Corollary 2.2]{Pop96disk}. It follows that $\mu _b (L^{\n}) := K_0 ^* V^{\n} K_0$ belongs to $CP (\mc{S} , \H )$. However, since $V$ is not an isometry, it is not obvious that $\mu _b \left( (I-zL^* ) ^{-1} (I - L w^* ) ^{-1} \right) =  K_0 ^* (I-zV^*)^{-1} (I-Vw^* )^{-1} K_0$. Indeed, given any $d$-contraction $T$ on $\H$, the relation $\phi _T ( (L ^\alpha ) ^* L ^\beta ) = (T ^\alpha ) ^* T^\beta$ holds for all $\alpha ,\beta \in \F ^d$ if and only if $T$ is an isometry.

    Let $W$ be the minimal isometric dilation of $V = V^b$ on $\K _b \supset \L (b)$. Then for any $\n \in \N ^d$, $\mu _b (L ^{\n} ) = K_0 ^* W^{\n} K_0$.
Since $W\supseteq V$ extends $V$, we have that for any $z, w \in \B ^d$,
\ba \mu _b \left( (I-zL^* ) ^{-1} (I - L w^* ) ^{-1} \right) & = & K_0 ^* (I-zW^*) ^{-1} (I-Ww^*)^{-1} K_0
\nn \\
& =& K_z ^* K_w \quad \quad \mbox{by Lemma \ref{RKext}} \nn \\
& = & K^b (z,w) \nn \\
& = & K_0 ^* (I -zV^* ) ^{-1} (I - V w^* ) ^{-1} K_0; \quad \quad \mbox{by Lemma \ref{RKext} again}. \nn \ea  The first line in the above equation
can be verified using that both $W$ and $L$ are row isometries so that their component operators obey $L_k ^* L_j = \delta _{kj} I $ and $\delta _{kj} I = W_k ^* W_j$.
\end{proof}

Observe that \ba 2 K(z,0) = H (z) + H(0)^* & = & H_b (z) + \re{H_b (0)} -i \im{H_b (0)} \nn \\
&  = & H_b (z) + K(0,0) -i\im{H_b (0)}. \nn \ea
One can then express the Herglotz-Schur function $H_b (z)$ as 
\ba  H(z) & = & 2K(z, 0 ) -  K(0,0) +i \im{H (0)} \nn \\
& = & K_0 ^* \left( 2(I -zV^* ) ^{-1} -  I \right) K_0 + i \im{H(0)} \nn \\
& = & K_0 ^* (I-zV^*) ^{-1} (I +zV^*) K_0 +i \im{H(0)} \nn \\
& = & \mu _b \left( (I-zL^* )^{-1} (I + zL^*) \right) + i \im{H (0)}. \nn \ea

The span of all $(I -Lz^* ) ^{-1} (I +Lz^*) $ for $z \in \B ^d$ is dense in $\mc{S}$ so that $\mu _b$ is uniquely defined by this formula. Moreover, given any $\mu \in CP (\mc{S} , \H )$, it is not hard to see that
$$ H(z) := \mu \left( (I-zL^*) ^{-1} (I+Lz^*) \right) \in \L (\H), $$ defines a pure Herglotz-Schur function so that the map $b \mapsto \mu _b$ is bijective (modulo imaginary constants).

We then have:
\begin{thm}{ (Noncommutative Herglotz formula)}
For any $b \in \left[ H^\infty _d \otimes \L (\H ) \right] _1$ there is a unique $CP$ map $\mu _b : \mc{S} +\mc{S} ^* \rightarrow \L (\H )$
such that
\be H_b (z) = (1+b(z))(1-b(z)) ^{-1} = \mu _b \left( (I + zL^* ) (I- zL^*) ^{-1} \right) +i \im{H_b (0)}. \label{HRF} \ee  The map $b \mapsto \mu _b$
is a bijection (modulo $\im{H_b (0)} \in \L (\H)$).
\end{thm}
Observe that $(I-zL^* ) ^{-1} (I+zL^*)$ is a direct analogue of the Herglotz integral kernel appearing in the classical Herglotz representation formula (\ref{classHerg}).
The above defines bijections between:
\bn
    \item purely contractive elements in the closed unit ball of $H^\infty _d \otimes \L (\H)$.
    \item the (purely contractive) Herglotz-Schur class of $\L (\H )$ valued functions $H$ on $\B ^d$.
    \item completely positive maps from the operator system $\mc{S} + \mc{S} ^*$ into $\L (\H )$ (modulo imaginary constants).
\en

\subsection{Noncommutative Fantappi\`{e} transform}

Given $\mu \in \mr{CP} ( \mc{S}  , \H  )$, we have by the previous section that $\mu = \mu _b$ for some
$b \in [ H^\infty _d \otimes \L (\H )] _1$.

We can use $\mu$ to construct a Stinespring-GNS type Hilbert space which (following \cite{Jur2014AC}) we will call $P^2 (\mu )$. This space will play the role
of the `closure of the analytic polynomials in the measure space $L^2 (\mu)$' in this several variable case.
First consider the algebraic tensor product $\mc{S} \otimes \H$ equipped with the sesquilinear form:
$$ \ip{ p \otimes h }{q \otimes g} _\mu := \ip{h}{\mu (p^* q) g}, $$ for any $p,q \in \mc{S}$ and any $h,g \in \H$. This is well-defined since the operator space $\mc{S}$ has the semi-Dirichlet property $\mc{S} ^* \mc{S} \subset \ov{ \mc{S}  + \mc{S} ^*} ^{\| \cdot \| }$ \cite{Jur2014AC}. As in the usual proof
of Stinespring's dilation theorem, the fact that $\mu$ is completely positive ensures that $\ip{\cdot}{\cdot}_\mu $ is a pre-inner product
on $\mc{S} \otimes \H$. If $N _\mu$ is the set of all elements $r \in \mc{S} \otimes \H$ which have zero length, $ \ip{r}{r} _\mu = 0$,
then $N _\mu$ is a subspace of $\mc{S} \otimes \H$ and $\ip{\cdot}{\cdot} _\mu$ defines an inner product on the quotient
$$ \frac{\mc{S} \otimes \H}{N_\mu}. $$ Let $P^2 (\mu )$ be the Hilbert space completion of this inner product space. We will
also use the notation $P^2 (b)$ for $P^2 (\mu )$ when $\mu = \mu _b$.

We will define unweighted and weighted versions of a Cauchy (Fantappi\'{e}) transform, implementing a unitary equlivalence between $P^2(b)$ and $\L(b), K(b)$ respectively.
For $z\in\B^d$ define the non-commutative Cauchy kernel
\begin{equation*}
  C_z(L): = (I -  Lz^* ) ^{-1} \in \mc{S}.
\end{equation*}
Then for any $h,g \in \H$,
\ba
  \ip{C_z(L)\otimes h}{C_w(L)\otimes g}_\mu & = & \ip{h}{\mu \left( (I -  zL^* ) ^{-1} (I- Lw^* ) ^{-1} \right)g}_\H \nn \\
  & = & \ip{h}{K^b (z,w) g} _\H \quad \quad \mbox{by Proposition \ref{CPmap}} \nn \\
  & = & \ip{ K^b _z h}{K^b _w g} _{\L (b)}. \nn \ea
Thus, the {\em Cauchy transform} on $P^2(\mu)$, defined by
\be \label{ncCT}
  (\mc{C}_b (p\otimes h)(z) := \mu(C_z(L)^*p(L))h,
\ee
is unitary from $P^2(b)$ onto $\L(b)$.

Since the multiplication $F(z)\mapsto (I-b(z))F(z) =U_b (z) ^{-1} F(z)$ is a unitary map from $\L(b)$ onto $K(b)$ (Lemma \ref{ontoisomult}), we obtain the {\em weighted Cauchy transform} by composition
\begin{equation*}
  p\otimes h \mapsto (I-b(z)) \mu(C_z(L)^*p(L))h.
\end{equation*}
This defines a unitary from $P^2(b)$ onto $K(b)$.

\begin{thm}{ (Noncommutative Fantappi\`{e} transform)}
Given any $CP$ map $\mu = \mu _b : \mc{S} +\mc{S} ^* \rightarrow \L (\H)$, the formula
\be \left( \mc{F} _b (p \otimes h) \right) (z)  = (I-b(z)) \mu \left( (I-L^*z)^{-1} p(L) \right) h, \label{fantap} \ee
defines a unitary transformation of $P^2 (b )$ onto $K (b)$.
\end{thm}

Comparison to Theorem \ref{classCT} shows this is a natural generalization of the single variable fact. For this reason we view $P^2 (b)$ as the several-variable analogue of the `closure of the analytic polynomials'. Under the unitary transformation induced by $\mc{C} _b$, the partial isometry $V:\L(b) \otimes \C ^d \to \L (b)$ is conjugate to the map $\hat{V}:P^2(\mu) \otimes \C ^d \to P^2(\mu)$ defined by
\be  \hat {V}( w^*(I-w^*L)^{-1}\otimes h + N_\mu) = w^*L(I-w^*L)^{-1}\otimes h +N_\mu, \label{hatV} \ee
and extended to be $0$ on $\mc{C} ^* \ker{V}$.

\section{Extensions to the Cuntz-Toeplitz operator system}\label{extensions}

Let $\mu : \mc{S} + \mc{S} ^* \rightarrow \L (\H )$ be a CP map. Recall that we use the notation $CP (\mc{S} , \H )$ for the set of all completely positive maps from $\mc{S} + \mc{S} ^*$ into $\L (\H )$ and
$CP (\mc{A} , \H )$ for all CP maps on $\mc{A} + \mc{A} ^*$ into $\L (\H )$. Further recall that any such CP map $\mu \in CP (\mc{S} , \H )$ is equal to $\mu _b$ for a unique $b $ in the operator-valued Schur class of $H^2 _d \otimes \H$.

It will be convenient to briefly review the Stinespring dilation of a CP map $\phi : \mc{A} + \mc{A} ^* \rightarrow \L (\H)$ to a unital completely isometric isomorphism of the operator algebra $\mc{A}$ into $\L (\H )$. First, the Stinespring GNS Hilbert space $Q^2 (\phi )$ is constructed in the same way that we constructed $P^2 (\mu )$. As before, the construction of $Q^2 (\phi )$ relies on the semi-Dirichlet property of $\mc{A}$: $\mc{A} ^* \mc{A} \subset \ov{\mc{A} + \mc{A} ^*} ^{\| \cdot \|}$.

Consider the algebraic tensor product $\mc{A} \otimes \H$, and the
sesquilinear form $\ip{\cdot}{\cdot} _\phi$ on $\mc{A} \otimes \H$
defined on elementary tensors by
$$ \ip{a\otimes h}{b \otimes g} _\phi := \ip{h}{\phi (a^* b) g}
_\H. $$ and extend linearly. Again, the facts that $\phi$ is CP and $\mc{A}$ has the semi-Dirichlet property implies that this is a well-defined pre-inner product obeying
the Cauchy-Schwarz inequality. Taking the quotient of $\mc{A} \otimes \H$ by the subspace
$$ N_\phi := \{ x\in \mc{A}\otimes \H | \ \ip{x}{x} _\phi = 0 \}, $$ yields an inner product space whose completion is denoted by $Q^2 (\phi )$.

\begin{remark}
Since $\phi$ extends $\mu$, the map $p \otimes h + N_\mu \mapsto p \otimes h + N _\phi$ from $P^2 (\mu )$ into $Q^2 (\phi )$ is a well-defined isometry, and hence one can view $P^2 (\mu )$ as a subspace of $Q^2 (\phi )$.  We will often identify $P^2 (\mu )$ with its image under this isometry in $Q^2 (\phi )$,
    and we will sometimes write $P^2 (\phi )$ for the embedding of $P^2 (\mu )$ in $Q^2 (\phi )$.
\end{remark}

One can construct a Stinespring dilation of $\phi$, $\pi _\phi : \mc{A} \rightarrow \L (Q^2 (\phi))$ as in the usual proof of Stinespring's theorem.
Namely, for any $a \in \mc{A}$, let $\pi _\phi (a) \in \L (Q^2 (\phi ))$ be defined by left multiplication:
\ba \pi _\phi (a) (b \otimes h + N_\phi) & := & (L_a \otimes I) (b \otimes h + N_\phi) \nn \\
 & := & ab \otimes h + N_\phi. \nn \ea It is easy to check this is a well-defined, contractive and unital linear map. Repeating the construction of $Q^2 (\phi )$ for the matrix operator algebras $\mc{A} \otimes \C ^{k \times k }$,
$$Q^2 (\phi) \otimes \C ^k \simeq \left[ \frac{ \left( \mc{A} \otimes \C ^{k \times k} \right) \otimes \left( \H \otimes \C ^k \right) }{ N _{\phi ^{(k)} }} \right], $$
where $\phi ^{(k)} : \mc{A} \otimes \C ^{k \times k} + \mc{A} ^* \otimes \C ^{k\times k} \rightarrow \L (\H \otimes \C ^k )$ is the $k$-fold ampliation of $\phi$, shows that $\pi _\phi$ is a completely contractive unital homomorphism of $\mc{A}$ into $\L (Q^2 (\phi ) )$. The square brackets in the above formula denote completion. Recall here that for $[a_{ij}] \in \mc{A} \otimes \C ^{k\times k}$ one defines the ampliation $\phi ^{(k)}$ by
$\phi ^{(k)} \left( [a_{ij} ] \right) = [ \phi (a_{ij} ) ]$. The map $\pi _\phi$ is a dilation of $\phi$: if $ [I \otimes ]_\phi : \H \rightarrow Q^2 (\phi ), $ is the linear map defined by
$ [I\otimes ]_\phi h  = I \otimes h + N _\phi, $ then $$ \phi (a) = [I\otimes ]_\phi ^* \pi _\phi (a) [I\otimes ]_\phi; \quad \quad a \in \mc{A},$$ and $[I\otimes ]_\phi $ has norm
$$ \| [I\otimes ] _\phi \| ^2 = \| \phi ( I ) \|. $$ The bounded linear map $[I\otimes ]_\phi$ is an isometry if and only if $\phi$ is unital (if and only if $\mu$ is unital).

\begin{remark} \label{SSstarextend}
    By definition, $\pi _\phi (L _k ) ^* \pi _\phi (L _j ) = \delta _{kj} I$ so that $\pi _\phi (L)$ is a row-isometry, and Lemma \ref{starextend} implies that
$\pi _\phi : \mc{A} \rightarrow \L (Q^2 (\phi ) )$ is a completely isometric unital homomorphism which obeys $\pi _\phi (a^* c ) = \pi _\phi (a) ^* \pi _\phi (c)$
for all $a,c \in \mc{A}$, and is $*$-extendible to a representation $\Pi _\phi$ of the Cuntz-Toeplitz $C^*$-algebra $\mc{E} = C^* (\mc{A} )$.
\end{remark}

\subsection{Tight extensions}

Given any CP $\mu : \mc{S} + \mc{S} ^* \rightarrow \L (\H )$ as above, let $\phi : \mc{A} + \mc{A} ^* \rightarrow \L (\H )$ be a CP extension of $\mu$. In this section
we will show that one can construct an extension $\phi$ of $\mu$ which is \emph{tight} in the sense of \cite{Jur2014AC}.

Recall that $P^2 (\mu )$ can be viewed as a subspace $P^2 (\phi)$ of $Q^2 (\phi )$. Let $P$ denote the orthogonal projection onto this subspace.
Also define the spaces
$$P^2 _0 (\phi ) := \bigvee _{\n \neq 0} L^{\n } \otimes h \subset P^2 (\phi ), $$ with projection $P_0$ and
$$Q^2 _0 (\phi ) := \bigvee _{\alpha \neq \emptyset} L^\alpha \otimes h \subset Q^2 (\phi ), $$ with orthogonal projection $Q_0$ and note that
$P^2 _0 (\phi) \subset Q^2 _0 (\phi)$ so that $P_0 \leq Q_0$. The space $P^2 _0 (\phi )$ is our analogue of the closed linear span of the non-constant analytic monomials in the single variable theory. Also note that $Q^2 _0 (\phi )$ is invariant for $\pi _\phi (L) =: R$.

Define the $d$-contraction $S$ on $P^2 _0 (\phi )$ by compression of $\pi _\phi (L)$:
$$ S = P_0 \pi _\phi (L) P _0 = P_0 R P_0. $$ Similarly let
$$ T := \pi _\phi (L) | _{Q^2 _0 (\phi )} = R |_{Q^2 _0 (\phi )}, $$ this is a $d$-isometry using that $Q^2 _0 (\phi )$ is invariant for $\pi _\phi (L)$.

\begin{defn}
The extension $\phi$ of $\mu$ is \emph{tight} if $T$ is a dilation of $S$ \cite[Definition 3.2]{Jur2014AC}.
\end{defn}

Since $P_0 \leq Q_0$, it will follow that $\phi$ is tight if we can show that $P_0$ is semi-invariant for $\pi _\phi$ \cite{Sar1965semi}, that is, if for any $k,j \in \{ 1, ... , d \}$,
\ba P_0 R_k P_0 R_j P_0 & = & P_0 \pi _\phi (L_k) P_0 \pi _\phi (L_j) P_0 \nonumber \\
& = & P_0 \pi _\phi (L_k L_j ) P_0 = P_0 R_k R_j P_0. \nonumber \ea

We now define a natural CP extension $\nu : \mc{A} + \mc{A} ^* \rightarrow \L (\H )$ of $\mu$: By Theorem \ref{HRF} there is a unique $b \in [ H^\infty _d \otimes \L (\H ) ] _1$ such that $\mu = \mu _b$. Recall that by Proposition \ref{CPmap},
\ba \mu _b \left( (I-zL^*)^{-1} (I-Lw^* ) ^{-1} \right) & = & (K_0 ^b )^* (I-zV^* ) ^{-1} (I-Vw^* ) ^{-1} K_0 ^b \nn \\
& = & K^b (z,w), \nn \ea and
$$ \mu _b (L ^{\n} ) = (K_0 ^b )^* V^{\n} K_0 ^b; \quad \quad \n \in \N ^d. $$
\begin{defn} \label{tightcand}
    Define $\nu \in CP (\mc{A} , \H )$ by $$ \nu ( L ^\alpha ) := K_0 ^* V^\alpha K_0; \quad \quad \alpha \in \F ^d. $$
\end{defn}
We will show that this $\nu$ is always the (unique) tight extension of $\mu$. If $\mu = \mu _b$ we will often write $\nu _b$ for the (tight) extension defined above. Similarly we will often write $Q^2 (b)$ for $Q^2 (\nu )$ and $\pi _b$ for
$\pi _{\nu}$.

\begin{remark} \label{cyclicspace}
As before, let $(W, \K _b )$ be the minimal isometric dilation of $V=V^b$ on $\K _b \supset \L (b)$. The space
$$ \K  = \bigvee _{\alpha \in \F ^d} W^\alpha K_0 h, $$ is equal to $\K _b$ since Lemma \ref{RKext} implies the right hand side contains $\L (b)$, and $\L (b)$ is cyclic for the minimal isometric dilation of $V$.
\end{remark}

\begin{prop}{ (Extended Cauchy Transform)} \label{extendedCT}
    Consider the linear map $\mc{C} _b : Q^2 (b) \rightarrow \K = \K _b$ defined by $$ \mc{C} _b (L^\alpha \otimes h ) = W^\alpha K_0 h.$$ Then $\mc{C} _b$ is unitary and
intertwines $ R = \pi _b (L)$; that is, $ \mc{C} _b \pi _b (L) =  W \mc{C} _b$.

The restriction of $\mc{C} _b$ to $P^2 (b)$ is the Cauchy transform onto $\L (b)$ from Section \ref{Herglotz}, $P^2 (b ) $ is co-invariant for $R = \pi _b (L) $ and  $R^* | _{P^2 (b)} = \hat{V} ^* =  \mc{C} _b ^* V^* \mc{C} _b$ is
a partial isometry.
\end{prop}

\begin{proof}
    This is a straightforward computation: If $L^\alpha \otimes h$, $L^\beta \otimes g$ belong to $Q^2 (\nu )$ then
\ba \ip{ L^\alpha \otimes h}{L^\beta \otimes g} _\nu & = & \ip{h}{\nu \left( (L ^\alpha ) ^*  L ^ \beta \right) g } _\H \nonumber \\
& = & \ip{W^\alpha K_0 h }{W^\beta K_0 g} _\K. \nonumber \ea
since $W$ dilates $\pi_b$.
The intertwining relation is also easily verified:
$$ \mc{C} \pi _b ( L^\gamma ) ( L^\alpha \otimes h ) = \mc{C} ( L ^{\gamma + \alpha } \otimes h )  = W ^\gamma \mc{C}  (L^\alpha \otimes h). $$
The map $\mc{C}$ is onto $\K = \K _b$ by Remark \ref{cyclicspace}, and it is clear that its restriction to $P^2 (b)$ is the Cauchy transform onto $\L (b)$ defined in equation (\ref{ncCT}).

Since $\L (b)$ is co-invariant for $W$, the minimal isometric dilation of $V$, the intertwining relation shows $P^2 (b)$ is co-invariant for $R=\pi _b (L)$. By the intertwining relationship, the fact that $W^* | _{\L (b)} = V^*$, and co-invariance, $R^* | _{P^2 (b)} = \mc{C} ^* W^* \mc{C} | _{P^2 (b) } = \mc{C} ^* V^* \mc{C} =\hat{V} ^*,$ as defined in equation (\ref{hatV}).
\end{proof}

The above result also shows that $P_b = \mc{C} P \mc{C} ^*$, where $P_b$ projects the space
$\K = \K ^b$ of the minimal isometric dilation $W$ onto $\L (b)$ so that $P : Q^2 (b) \rightarrow P^2 (b)$ is a co-invariant projection for $\pi _b (L)$.

The next lemma records the relationships between the projections $P_0$ and $Q_0$ in $P^2(b)\subset Q^2(b)$.
\begin{lemma} \label{intersect}
With notations as above,
\begin{itemize}
\item[i)] $Q_0=RR^*$,
\item[ii)] $P_0 = \mc C_b^* VV^* \mc C_b$,
\item[iii)] $P_0=Q_0P=PQ_0$.
\end{itemize}
\end{lemma}
\begin{proof}
    First $\ran{Q_0} = Q^2 _0 (b ) := \bigvee _{\alpha \neq \emptyset} L^\alpha \otimes \H = \ran{ \pi _b (L) } = \ran{R}$. (Actually this holds not only for $Q^2(b):=Q^2(\nu)$, but for any $Q^2(\phi)$ where $\phi$ is a CP extension of $\mu$. ) Since $R$ is an isometry, we conclude $Q_0=RR^*$. This proves (i).

Items (ii) and (iii) will be proven simultaneously. Consider first $\wt{Q} _0 := \mc{C} Q_0 \mc{C}^* = \mc{C} RR^* \mc{C} = W W ^*$, the projection onto $\ran{W}$. Then,
\ba  \wt{Q_0} P_b & = & W W ^* P_b = W V^* P_b \nonumber \\
& = & W (V^*V) V^* P_b = V (V^* V ) V^* P_b \nonumber \\
& = & V V^* P_b = VV^*. \nonumber \ea In the above, the identity $W (V^* V) = V (V^*V)$ holds since
$W^*$ is a contractive extension of $V^*$ (and hence $W$ is a contractive extension of the partial $d-$isometry $V$, see Lemma \ref{contractext}).
This shows that $\wt{Q} _0 P_b $ and hence $Q_0 P$ is a projection.
By the von Neumann alternating projection formula $Q_0P$ is the projection onto $Q^2 _0 (b) \cap P^2 (b)$. Now clearly
$P^2 _0 (b)  \subset  Q^2 _0 (b) \cap P^2 (b )$ so that $P_0 \leq Q_0 P$. Conversely $Q_0 P = \mc{C}_b ^* V V ^* \mc{C}_b = \hat{V} \hat{V} ^*$ is unitarily equivalent to
$V V ^*$, the projection onto $\ran{V} = \bigvee _{z \in \B ^d} (K_z - K_0) \H.$ For any $h \in \H$, Lemma \ref{RKext} implies
\ba (K_z - K_0) h & = & \left( (I- V z^* ) ^{-1} - I  \right) K_0 h \nonumber \\
& =  & \left( (I-Wz^* )^{-1} - I \right) K_0 h \nonumber \\
& =& \sum _{k=1} ^\infty (Wz^*) ^k K_0 h. \nonumber \ea
Since $Q_0 P$ projects onto $\mc{C}_b ^* \ran{V}$ and $$ \mc{C}_b^* \sum _{k=1} ^\infty (Wz^*) ^k K_0 h  =
\sum _{k=1} ^\infty (Lz^*) ^k \otimes h \in P^2 _0 (\nu), $$ it follows that $Q_0 P \leq P_0$. We conclude that $P_0 = Q_0 P$.
\end{proof}

\begin{cor}
The projection $P_0$ of $ Q^2 (b)$ onto $P^2 _0 (b)$ is semi-invariant for $R =  \pi _b (L)$ and $\nu =\nu _b \in CP (\mc{A} , \H )$
is a tight extension of $\mu = \mu _b \in CP (\mc{S} , \H )$.
\end{cor}
\begin{proof}
  As before, let $R:= \pi _\nu (L)$. Then,
\ba P_0 R_k P_0 R_j P_0 & = & Q_0 P R_k PQ_0 R_j Q_0 P \nonumber \\
& =& Q_0 P R_k Q_0 R_j Q_0 P \quad \mbox{ (By co-invariance of P) } \nonumber \\
& =& Q_0 P R_k R_j Q_0 P \quad \mbox{(By invariance of $Q_0$)} \nonumber \\
& = & P_0 R_k R_j P_0. \nonumber \ea

To see that $T := \pi _\nu (L) Q_0 = R Q_0$ is a dilation of $S := P_0 R P_0$, apply the semi-invariance of $P_0$.
\ba S_k S_j & = & P_0 T_k P_0 T_j P_0 \nonumber \\
& = &  P_0 R_k P_0 R_j P_0 = P_0 R_k R_j P_0 \nonumber \\
& = &  P_0 R_k R_j Q_0 P_0 = P_0 R_k Q_0 R_j Q_0 P_0 \nonumber \\
& = & P_0 T_k T_j P_0. \ea This proves that $\nu$ is tight.
\end{proof}

\begin{remark}
    The tight extension of any $\mu \in CP (\mc{S} ,\H )$ is necessarily unique \cite[Theorem 3.5]{Jur2014AC}. (The result there is stated and proved only for scalar-valued $\mu$, but the proof works {\it mutatis mutandis} for general CP maps.)
\end{remark}

\begin{cor}
    If $\mu : \mc{S} + \mc{S} ^* \rightarrow \L (\H ) $ has a unique extension $\phi : \mc{A} +\mc{A^*} \rightarrow \L (\H )$ then $\phi = \nu _b$ is tight.
\end{cor}

\subsection{Extensions of $\mu$ and cyclic isometric extensions of $V$}

In this subsection we show that given $\mu \in CP (\mc{S} , \H )$, the set of all $\phi \in CP (\mc{A} , \H )$ extending $\mu$ is naturally parametrized
by equivalence classes of cyclic isometric extensions of the partial $d-$isometry $V =V^b$ acting on the Herglotz space $\L (b)$.

\begin{defn}
    Let $\Ext{V}$ be the set of all $d-$isometric extensions $D \supseteq V$ acting on a Hilbert space $\J \supset \L (b)$ so that $K_0 \H$ is cyclic for $D$.
Given $D_1 , D_2 \in \Ext{V}$, we say $D_1 \simeq _b D_2$ if $D_1, D_2$ are unitarily equivalent via an isometry which restricts to the identity on $\L (b)$.
\end{defn}

\begin{lemma} \label{phiD}
    Given any row contractive extension $D \supseteq V$ on $\J \supset \L (b)$, the $CP$ map $\phi _D \in CP (\mc{A}, \H )$ defined by
$ \phi _D (L^\alpha) := K_0 ^* D^\alpha K_0 $ is an extension of $\mu$.
\end{lemma}
\begin{proof}
For any $D \supseteq V$, we can define a $CP$ map by the formula in the statement above:
$$ \phi _D (L^\alpha) := K_0 ^* D^\alpha K_0. $$ Recall that $ \mc{S} = \bigvee _{z \in \B ^d} (I-z^*L) ^{-1}.$
Then,
\ba  \phi_D \left( (I-z^*L) ^{-1} \right) & = & K_0 ^* (I-z^*D)^{-1} K_0 \nonumber \\
& =& K ^b (0 ,z ) \nonumber \\
& =& \mu _b \left( (I-z^*L) ^{-1} \right) \ea where we used Lemma \ref{RKext} in the first line and Proposition \ref{CPmap} in the third line above.
\end{proof}

\begin{prop}
    Given any $D_1 , D_2 \in \Ext{V}$ defined on $\J _k \supset \L (b)$ we have that $\phi _1 := \phi _{D_1} = \phi _{D_2} =: \phi _2 \in CP (\mc{A} , \H )$ if and only if
$D_1 \simeq _b D_2$.
\end{prop}

\begin{proof}
   If $D_1 \simeq _b D_2$, it is obvious that $\phi _1 = \phi _2$. Conversely if $\phi = \phi _1 = \phi _2$ suppose that $D_k$ are $d-$isometries on $\J _k \supset \L (b)$. By assumption, $K_0 \H$ is cyclic for each $D_k$ so that $$ \J _k = \bigvee _{\alpha \in \F ^d} D_k ^\alpha K_0 \H. $$ Define a linear map $U : \J _1 \rightarrow \J _2$ by $U D_1 ^\alpha K_0 h := D_2 ^\alpha K_0 h$. This
is onto and it is an isometry:
\ba \ip{U D_1 ^\alpha K_0 h}{U D_1 ^\beta K_0 g} _{\J _2} & = & \ip{h}{K_0 ^* (D_2 ^\alpha ) ^*  D_2 ^\beta K_0 g} _\H \nn \\
& = & \ip{h}{\phi ( (L ^\alpha ) ^* L ^\beta) g}_\H \nn \\
& = & \ip{h}{K_0 ^* (D_1 ^\alpha ) ^* D_1 ^\beta K_0 g} _\H \nn \\
& =&  \ip{D_1 ^\alpha K_0 h}{ D_1 ^\beta K_0 g} _{\J _1}. \nn \ea
It is clear that $U D_1 = D_2 U$ and $U$ restricts to the identity on $\L (b)$ since $D_1, D_2 \in \Ext{V}$ implies that
\ba U K_z h & = & U (I-D_1 z^* ) ^{-1} K_0 h \nn \\
& = & (I-D_2 z^* ) ^{-1} K_0 h \nn \\
& = & K_z h, \nn \ea by Lemma \ref{RKext}.
\end{proof}

\begin{cor}
    There is a bijection between CP extensions $\phi$ of $\mu _b \in CP (\mc{S} , \H )$ to $\mc{A} + \mc{A} ^*$ and $\simeq _b$ equivalence classes of $\Ext{V ^b}$. \label{extchar}
\end{cor}

\begin{proof}
    It remains to show that the map $D \in \Ext{V} \mapsto \phi _D$ is onto the set of all CP extensions of $\mu _b$. If $\phi \in CP (\mc{A} ,\H)$ extends
$\mu _b \in CP (\mc{S} , \H )$ then we can identify (using the inverse Cauchy transform) $V = V^b$ with the partial $d-$isometry $\hat{V}$ acting on $P^2 (\phi) $ with initial and final spaces:
$$ \ker{\hat{V}} ^\perp = \bigvee   z^*(I-L z^* ) ^{-1} \otimes \H, \quad  \quad \ran{\hat{V}} := \bigvee (Lz^*) (I-Lz^* ) ^{-1} \otimes \H, $$ and
$$ \hat{V} \left( z^* (I-Lz^*)^{-1} \otimes h \right) = (Lz^*) (I-Lz^* ) ^{-1} \otimes h. $$ Since $\phi$ extends $\mu$, $P^2 (\mu)$ embeds
isometrically as $P^2 (\phi )$ into $Q^2 (\phi )$, and is clear from the definition of $\hat{V}$ that $\pi _\phi (L)$ is $d-$isometric extension of $\hat{V}$ with cyclic space $[I\otimes ]_\phi \H =  I \otimes \H + N _\phi$. Since the Cauchy transform
$\mc{C} : P^2 (\mu )  \rightarrow \L (b)$ obeys $\mc{C} (I \otimes h + N_\mu ) = K_0 h$, it follows that we can construct a $d-$isometric extension $D\simeq \pi _\phi (L) $ of $V^b \simeq \hat{V}$ with cyclic subspace $K_0 \H$ so that $ \phi (L ^\alpha ) = K_0 ^* D ^\alpha K_0. $
\end{proof}

\begin{thm}
    Assume that $d>1$. A $CP$ map $\mu =\mu_b : \mc{S} + \mc{S} ^* \rightarrow \L (\H)$ has a
unique extension to $\mc{A} + \mc{A} ^*$ if and only if $V$ is a co-isometry.
\label{nouniext}
\end{thm}

\begin{proof}
If $V$ is a co-isometry, then given any $D \in \Ext{V}$ we have that $V^* \subseteq D^*$, and since $\ran{V} = \L (b)$ it follows that $\L (b)$ is co-invariant for $D$ and $V^* K_z h = D^* K_z h$ for all $z \in \B ^d $ and $h \in \H$.
By Corollary \ref{extchar}, any CP extension $\phi $ of $\mu = \mu _b$ has the form $\phi  = \phi _D$ for some $D \in \Ext{V}$. Namely $\phi (L ^\alpha ) = K_0 ^* D ^\alpha K_0$. Since $D^* | _{\L (b)} = V^*$, it follows that $\phi  =  \phi _V =  \nu _b$, the tight extension of $\mu _b$.

Conversely suppose that $V=V^b$ is not co-isometric so that $\ran{V} \neq \L (b)$. In this case (using the assumption that $d>1$) one can construct a non-trivial $d-$contractive extension $D \supseteq V$ acting on $\L (b)$.
This relies on the fact that for $d>1$, and any $b \in [H^\infty _d \otimes \L (\H ) ] _1$, $\ker{V^b} \neq \{ 0 \}$ is never trivial (see Remark \ref{ntkernremark} below). Even though $D$ is generally not an isometry, we can define a completely positive map $\phi _D (L ^\alpha ) := K_0 ^* D ^\alpha K _0$ extending $\mu _b$ as in Lemma \ref{phiD}.

We claim that there exists an $h \in \H$ such that $D ^* K_0 h \neq V^* K_0 h$. Otherwise since $V^* \subseteq D ^*$  it would follow that
$D^* (K_z - K _0 ) h = V^* (K_z - K_0 ) h $ for all $h \in \H$  (recall here that $\ran{V} = \bigvee (K_z - K_0 ) \H$). Hence if $D ^* K_0 h = V^* K_0 h $ for all $h \in \H$ it follows that $D ^* K_z h = V^* K_z h $ for all $h \in \H$ and all $z \in \B ^d$. This would prove that $D=V$ (since $D^*$ acts on $\L (b)$).

It follows that if $D \neq V$, then there is an $h \in \H$ such that $D ^* K_0 h \neq V^* K_0 h$, so that there is a $1 \leq k \leq d$, $z \in \B ^d$ and
$g \in \H$ such that
$$ \ip{D^* _k K_0 h}{K_z g} _b \neq \ip{V^* _k K_0 h}{K_z g} _b.$$ By Lemma \ref{RKext},
\ba \ip{D^* _k K_0 h}{K_z g} & = &  \ip{h}{K_0 ^* D_k (I-z^*D) ^{-1} K_0 g} \nonumber \\
& = & \ip{h}{\phi _D \left( L_k (I-z^*L )^{-1} \right) g}. \nonumber \ea Similarly,
$$ \ip{V_k^* K_0 h}{K_z g} = \ip{h}{\phi _V \left( L_k (I-z^*L )^{-1} \right) g}. $$
This shows $\phi _D \neq \phi _V = \nu _b$.
\end{proof}

\begin{remark} \label{ntkernremark}
    When $d>1$, the kernel of the partial $d-$isometry $V^b$ on the Herglotz space $\L (b)$ is never trivial. We will present two (partial) proofs using the theory of solutions to the Gleason problem in $K(b)$
that will be developed in Section \ref{Gleasonsection}, see Proposition \ref{ntkernel} and Proposition \ref{ntkernel2}. The first  proof shows that $V^b$ has non-trivial kernel if there is a contractive solution $B$ to the Gleason problem in $K (b)$, $B \in \L ( \H , K(b) \otimes \C ^d )$ such that the closed span of the ranges of its component operators is not all of $K(b) \otimes \C ^d$. Although this condition is very restrictive, a proof that this is not possible in the general vector-valued case remains elusive at this time. The second proof, an abstract argument using  the Cuntz relations, shows that $\ker{V^b} \neq \{ 0 \}$ whenever $b \in [H^\infty _d \otimes \L (\H ) ] _1$ and $\H$ is finite dimensional (and of course $d>1$).

There is a third argument that works in full generality.  This proof uses the non-commutative or free analogue of the theory developed here for the free multiplier algebra $F^\infty _d \otimes \L (\H )$ of the vector-valued full Fock space $F^2 _d \otimes \H$ over $\C ^d$.  Here the full Fock space $F ^2 _d$ can be viewed as a non-commutative or free formal reproducing kernel Hilbert space, and the weak operator topology closed
unital operator algebra, $F^\infty _d$, generated by the non-commutative or free shift $L$ (the left creation $d$-isometry) can be viewed as the formal multiplier algebra of this free formal RKHS $F^2 _d$ \cite{Ball2003rkhs,Ball2006Fock}. This proof is beyond the scope of this paper and will be presented in forthcoming work.
\end{remark}

\subsection{Quasi-extreme maps}

\begin{defn}
Following \cite[Definition 3.7]{Jur2014AC}, a $CP$ map $\mu : \mc{S} + \mc{S} ^* \rightarrow \L (H )$ is called \emph{quasi-extreme} if $P^2 _0 (\mu ) = P^2 (\mu )$. Any $CP$ map $\phi : \mc{A} +\mc{A} ^* \rightarrow \L (\H )$ is called quasi-extreme if its restriction to $\mc{S} +\mc{S} ^*$ is quasi-extreme. An element $b \in [H^\infty _d \otimes \L (\H) ]_1$ is said to be quasi-extreme if $\mu = \mu _b \in CP (\mc{S} , \H )$ is quasi-extreme.
\end{defn}

This concept of quasi-extremity is a natural analogue of the Szeg\"{o} approximation or analytic polynomial density property $L^2 (\mu ) = P^2 (\mu ) = P^2 _0 (\mu )$ from the single variable, scalar-valued case.

\begin{cor}
    A $CP$ map $\mu : \mc{S} +\mc{S} ^* \rightarrow \L (\H)$ is quasi-extreme if and only if $R = \pi _\nu (L)$ is a Cuntz unitary. Here $\nu$ is the tight
extension of $\mu$. \label{cuntzu}
\end{cor}

Recall that a row isometry $W \in \L (\H \otimes \C ^d , \H)$ is a \emph{Cuntz unitary} if it is also co-isometric, $W W^* = I _\H$.

\begin{proof}
    One direction is already proven in \cite[Proposition 3.10]{Jur2014AC}, namely if $\mu := \nu |_{\mc{S} +\mc{S} ^*}$ and $\mu = \mu _b$ is quasi-extreme
then $R$ is a row unitary. (Again the result there is stated only in the scalar case but the proof holds generally.) Conversely suppose that $R = \pi _b (L)$ is a Cuntz unitary. Then since $Q^2 _0 (b) = \ran{R}$ it follows that
$Q^2 _0 (b ) = Q^2 (b)$ so that $Q_0 = Q$. By Lemma \ref{intersect}, we then have that $P_0 = P Q_0 = PQ = P$ so that $P^2 _0 (b) = P^2 (b )$.
\end{proof}

\begin{cor}
    A $CP$ map $\mu \in CP (\mc{S}, \H)$ has a unique extension $\phi \in CP (\mc{A} , \H)$ if and only if
$\mu$ is quasi-extreme. In this case $\phi = \nu $, the tight extension of $\mu $. Equivalently $b \in [H^\infty _d \otimes \L (\H) ] _1$ is
quasi-extreme if and only if $V^b$ is a row co-isometry on $\L (b)$. \label{uniquext}
\end{cor}

\begin{proof}
    As shown in \cite[Theorem 3.8]{Jur2014AC}, if $\mu$ is quasi-extreme, it has a unique extension. Since $\mc{C}  P^2 _0 (b )$ is $\ran{V ^b}$, the converse follows immediately from Theorem \ref{nouniext}.
\end{proof}

\begin{thm} \label{containb}
    An element $b \in [H^\infty _d \otimes \L (\H ) ] _1 $ is not quasi-extreme if and only if there is a non-zero $h \in \H$ so that $b h \in K (b)$.
\end{thm}

This is a direct analogue of a classical fact \cite[IV-4,V-3]{Sarason-dB}, and generalizes \cite[Theorem 3.13]{Jur2014AC}.

\begin{proof}
    If $b$ is not quasi-extreme then the partial $d-$isometry $V$ on $\L (b )$ is not a co-isometry so that $\ran{V} \neq \L (b)$. As discussed in Section \ref{Herglotz}, if $F \in \L (b)$ is orthogonal to $\ran{V}$ then $F$ is a constant function,
$F(z) = F(0) =: f \in \H$ for all $z \in \B ^d$. Using the canonical unitary multiplier of $\L (b)$ onto $K (b)$ it follows that $(I-b) f \in K(b)$. However we also have that $k_0 ^b f \in K(b)$ and
$$ k_0 ^b  f = (I - b b(0) ^* ) f. $$ It follows that $b (I-b(0) ^* ) f \in K(b)$. Since we assume $b$ is purely contractive, $b(0)$ is a strict contraction and $h := (I-b(0) ^* ) f  \in \H$ is non-zero and $bh \in \K(b)$. (Here $bh = M_b k_0 h$, where $k_0$ is the Szeg\"{o} kernel map, could be zero if $k_0 h$ is in the kernel of $M_b$.)

The above argument is reversible: If $h \in \H$ is non-zero and $bh \in \K (b)$ then since $b(0)$ is a strict contraction we have $0 \neq f:= (I-b(0) ^* ) ^{-1} h \in \H$ and
$0 \neq k_0 ^b f = f - b b(0) ^* f \in \K (b)$. Then,
\ba k_0 ^b f & = & (I-b(0) ^* ) ^{-1} h - b b(0) ^* (I-b(0) ^* ) ^{-1} h \nn \\
& = & (I -b(0) ^* ) ^{-1} h - b (I-b(0) ^* ) ^{-1} h + bh \nn \\
& = & (I-b) f + b h \in \K (b). \nn \ea It follows that $0 \neq (I-b) f \in K(b)$ for a non-zero $f \in \H$ so that the constant function $F(z) = f$ is such that $F \in \L(b)$
and necessarily $F \perp \ran{V}$. Hence $V$ is not co-isometric and $b$ is not quasi-extreme.
\end{proof}

\begin{cor}
    If $b$ is quasi-extreme then so is $b \alpha ^*$ for any unitary $\alpha \in \L (\H )$. That is, $\mu _b \in CP (\mc{S} , \H)$ is quasi-extreme, if and only if all of the Aleksandrov-Clark CP maps $\mu _\alpha = \mu _{b\alpha ^*}$ are quasi-extreme.
\end{cor}

\begin{prop}
    Let $\mu _1 , \mu _2 \in CP (\mc{S} , \H )$ be such that $\mu _2$ is quasi-extreme and $\mu _2 \geq \mu _1$. Then $\mu _1 $ is also quasi-extreme.
\end{prop}

\begin{proof}
    Since $\mu _2 \geq \mu _1$, $P^2 (\mu _2 ) $ is contractively contained in $P^2 (\mu _1 )$. Also since $\mu _2 $ is quasi-extreme $P^2 _0 (\mu _2 ) = P^2 (\mu _2 )$. For any $I \otimes h + N_2 \in P^2 (\mu _2 )$ choose a sequence $H_n \in \bigvee _{\n \neq 0 } L ^{\n} \otimes \H $ (here $\bigvee$ is just linear span, not norm closed) so that $H_n \rightarrow I \otimes h + N_2$ in the norm of $P^2 (\mu _2 )$. If $E : P^2 (\mu _2 ) \rightarrow P^2 (\mu _1)$ is the contractive embedding then
$$ E : \bigvee _{\n \neq 0 } ( L ^{\n} \otimes \H + N_2 ) \mapsto P^2 _0 (\mu _1 ), $$ so that $G_n := EH_n \in P^2 _0 (\mu _1)$. Hence
 \ba \| G_n - (I \otimes h + N_1) \| _{\mu _1} & = & \| E \left( H_n  - (I \otimes h + N_2 ) \right) \| _{\mu _1 }  \nn \\
 & \leq & \| H_n - (I \otimes h + N_2 ) \| _{\mu _2 } \rightarrow 0, \nn \ea and it follows that $P^2 _0 (\mu _1 ) = P ^2 (\mu _1)$.
\end{proof}

\section{The Gleason Problem in $K(b)$}\label{Gleasonsection}

This section studies the Gleason problem in a deBranges-Rovnyak subspace of vector-valued Drury-Arveson space. A solution to the Gleason problem in a deBranges-Rovnyak space $K(b)$ is the appropriate generalization of (the adjoint of) the restriction of the backward shift to $K(b)$ in the classical single variable theory. Our results here refine and extend those of \cite[Section 4]{Jur2014AC} obtained in the scalar-valued setting.

Given $b \in [H^\infty _d \otimes \L (H) ]_1$, consider the deBranges-Rovnyak space $K(b) = \H ( k ^b )$, the reproducing kernel Hilbert space of $\H$-valued functions on $\B ^d$ with kernel function
$$ k^b (z, w) := \frac{I - b(z) b(w) ^*}{1 - w^* z} \in \L (\H ); \quad \quad z,w \in \B ^d. $$

When $d>1$, the appropriate analogue of the shift operator is Arveson's $d-$shift $S : H^2 _d \otimes \C ^d  \rightarrow H^2 _d$, a partial isometry from $d$ copies of Drury-Arveson space into itself whose
component operators commute and act as multiplication by the independent variables, \cite{ArvIII}:
$$ (S \mbf{F} ) (z) = S \begin{bmatrix} F_1 \\ \vdots \\ F_d \end{bmatrix}  (z) = z \mbf{F} (z) = z_1 F_1 (z) + ... z_d F_d (z). $$ In contrast with the classical ($d=1$) case, deBranges-Rovnyak subspaces of Drury-Arveson space are in general not co-invariant for the component operators of the $d$-shift \cite{Ball2008}. Instead, the appropriate replacement for the restricted backward shift in this setting is a solution to the \emph{Gleason problem} \cite{Glea1964,Alpay2002,Ball2008,Ball2010,Ball2011dBR}:

\begin{defn} \label{Gledef}
A row-operator $X \in \L ( K (b) \otimes \C ^d , K(b) )$ solves the \emph{Gleason problem} in $K(b)$ if
\be z (X^*f) (z) := z_1 (X_1 ^* f ) (z) + ... + z_d (X_d ^* f ) (z) = f (z) - f(0); \quad \quad \forall f \in K (b). \label{Gleasoldef} \ee We say that a Gleason solution $X$ is \emph{contractive} if $$ X X ^* \leq I - k_0 ^b (k_0 ^b) ^*,$$ and is \emph{extremal} if equality holds in the above.
\end{defn}
Solutions to the Gleason problem in the Herglotz space $\L (b)$ are defined similarly although we say that a Gleason solution for $\L (b)$ is contractive if it is simply a $d-$contraction. In the case where $d=1$ the unique solution to equation (\ref{Gleasoldef}) is the adjoint of the restriction of the backward shift $S^*$ to $K(b)$, so that adjoints of Gleason solutions are natural analogues of the restricted backward shift in the several variable setting. Many references define a Gleason solution for $K(b)$ as the adjoint of our definition above, we prefer to view a Gleason solution as a row contraction.  Contractive solutions $X$ to the Gleason problem in $K(b)$ always exist, although they are in general non-unique \cite{Ball2010}.

\begin{defn}
 A linear map $B \in \L (H , \K (b)  \otimes \C ^d )$, $B = \begin{bmatrix} b_1 \\ \vdots \\  b_d \end{bmatrix}$, $b_j \in \L (\H , K(b))$, $1 \leq j \leq d$, is a solution to the Gleason problem for $b \in [H^\infty _d \otimes \L (\H )] _1$ provided that
$$ b(z) - b(0) = z \cdot B (z) := \sum _{j=1} ^d z_j b_j (z).$$ We say that $B$ is a \emph{contractive} Gleason solution for $b$ if
$$ B^* B \leq I - b(0) ^* b(0),$$ and an \emph{extremal} Gleason solution for $b$ if equality holds in the above.
\end{defn}
Superscript and subscript $b$'s will be omitted for the remainder of this section when this is clear from context.
\begin{lemma}
    A $d-$contraction $X$ solves the Gleason problem in $K (b)$ if and only if
    $$ X z^* k_z = k_z -k_0; \quad \quad \mbox{or equivalently,} \quad \quad k_z = (I-Xz^* ) ^{-1} k_0. $$
\end{lemma}

An analogous statement holds in $\L (b)$ replacing $K$ by $k$. Note that the analogue of the above lemma in $\L (b)$ implies that $D$ is a contractive solution to the Gleason problem in $\L (b)$ if and only if $D \supseteq V$ is a contractive extension of the partial $d-$isometry $V = V^b$ on $\L (b)$.  In particular, if $b$ is quasi-extreme then $\L (b)$ has $V$ as its unique contractive Gleason solution.

\begin{thm}\label{Glechar}
    $X$ is a $d$-contractive solution to the Gleason problem in $K(b)$ if and only if
    $$ X^* k_w  = w^* k_w   - B b(w) ^*, $$ where $B \in \L ( \H , K (b) \otimes \C ^d)$ is a contractive Gleason solution for $b$.
This defines a surjection from contractive Gleason solutions $B$ for $b$ onto contractive Gleason solutions $X$ for $K(b)$, $B \mapsto X(B)$. If $B$ is extremal then so is $X(B)$.
\end{thm}

\begin{proof}
   First suppose that $X$ has the assumed form. Then,
\ba (z X^* k_w) (z) & = & z w^*  k_w (z) - zB(z) b(w) ^* \nn \\
& = & z w^*  k(z,w) - (b(z) - b(0) ) b(w) ^* \nn \\
& = & z w^* k(z,w) + (1 - b(z) b(w) ^* ) - (1-b(0) b(w) ^* ) \nn \\
& = & k(z,w) - k(0, w) = k_w(z) - k_w (0). \nn \ea This proves that $X$ is a solution to the Gleason problem in $K(b)$.

It remains to check that the assumption that $B$ is a contractive Gleason solution for $b$ implies $X$ is a contractive solution: For any $w \in \B ^d$,
\ba k_w ^* X X ^* k_w & = & (k_w ^* w - b(w) B^*) (w^* k_w -B b(w) ^* ) \nn \\
& = &k_w ^* w w^* k_w - w B(w) b(w) ^*  - b(w) B(w) ^* w^* + b(w) B(w) ^* B(w) b(w) ^*. \nn \\
 & = &   w w^* k(w,w) - (b(w) - b(0))b(w)^* - b(w)(b(w) ^* - b(0) ^* ) + b(w) B (w) ^* B (w) b(w ) ^* \nn \\
& \leq & w w^* k(w,w)  - (b(w) - b(0))b(w)^* - b(w)(b(w) ^* - b(0) ^* ) + b(w) (I- b(0) ^* b(0) ) b(w ) ^* \label{equal1} \\
& = & k(w,w) - k(w, 0) k (0,w) \nn \\
& = &  k_w ^* \left( I - k_0 k_0 ^* \right) k_w. \nn \ea  This proves that the contractivity condition
\be  X X ^* \leq I - k_0 k_0 ^*, \label{equal3} \ee holds on kernel maps $k_w$. Since the ranges of the kernel maps are dense in $K(b)$, it follows that $X$ is a contractive Gleason solution. If equality holds in equation (\ref{equal1}) then it holds in equation (\ref{equal3}). It follows that if $B$ is extremal then so is $X$.

Conversely suppose that $X$ is a contractive Gleason solution and for each $w \in \B^d$ define the map $A _w \in \L (\H , K(b) \otimes \C ^d)$ by
$$ A_w := w^* k_w -X^* k_w. $$ Define an $\L (\H)$-valued kernel function $k^A$ on $\B ^d$ by
$$ k^A (z,w) := A_z ^* A_w. $$ Some algebra similar to the first part of the proof shows
\ba A _z ^* A_w & = & k_z ^* zw^* k_w - 2k(z,w) + k(z,0) + k (0,w) +k_z ^* XX^* k_w \nn \\
& \leq & zw^* k(z,w) - 2k(z,w) + k(z,0) + k(0,w) + k(z,w) -k(z,0)k(0,w), \nn  \\
& = & - (I - b(z) b(w)^* ) + k(z,0) + k(0,w) - k(z,0)k(0,w) \nn \\
& = & b(z) (I - b(0) ^* b(0)) b(w) ^*, \nn \ea as positive kernel functions. Let $F(z) := b(z) \sqrt{ I - b(0) ^* b(0)}$, and
$k^F (z,w) := F(z) F(w) ^*$, then $k^A \leq k^F$ as positive $\L (\H )-$valued kernel functions on $\B ^d$.
Define the co-isometry $U_A : K(b) \otimes \C ^d \rightarrow \H (k^A )$ by
$$ U_A A_w := k^A _w. $$ This is a co-isometry with initial space $\bigvee _{z\in \B ^d} A_z \H$ since
$(k^A _z) ^* k_w ^A = k^A (z,w) = A_z ^* A_w$ by definition. Since $\H (k^A)$ is contractively contained in $\H (k^F)$, let $E : \H (k^A ) \rightarrow \H (k ^F)$ be the inclusion map. Then
$E^* k^F _w = k^A _w$ and
$$ A_w = U^* E^* k^F _w ; \quad \quad w \in \B^d. $$  Define the constant $\L (\H)$-valued kernel function
$$ \delta (z,w) := I - b(0) ^* b(0); \quad \quad z,w \in \B ^d.$$ It follows that $m_b : \H ( \delta ) \rightarrow \H (k^F )$, multiplication by $b(z)$, is a co-isometric multiplier of $\H (\delta )$ onto $\H (k^F)$. Since the kernel $\delta (z,w)$ is constant, the point evaluation maps obey $\delta _z = \delta _0$ for all $z \in \B^d$ and
\ba A_w & = & U^* E^* k_w ^F \nn \\
& = & U_A ^* E^* m_b m_b^* k_w ^F \nn \\
& =& U_A ^* E^* m_b \delta _w b(w) ^* \nn \\
& = & U_A ^* E^* m_b \delta _0 b(w) ^*. \nn \ea Define, $$ B:= U_A ^* E^* m_b \delta _0 \in \L (\H , K(b) \otimes \C ^d ), $$ this is independent of $w\in \B^d$. By construction
\ba X^* k_w & = & w^* k_w - A_w \nn \\
& = & w^* k_w - B b(w) ^*, \nn \ea and
\ba B^* B & = & \delta _0 ^* m_b^* E U_A U_A ^* E^* m_b \delta _0  \nn \\
& \leq & \delta (0,0) = I -b(0) ^* b(0). \nn \ea  This shows that if $B$ is a Gleason solution for $b$ then it is contractive in the sense of Definition \ref{Gledef}. To see that $B$ is a Gleason solution for $b$, calculate that
\ba z B(z) b(w) ^* & = & zw^* k(z,w) - k_z ^* z X^* k_w \nn \\
&= & zw^* k(z,w) - (k_z - k_0 ) ^* k_w \nn \\
& =& k(z,w) - (I - b(z) b(w) ^*) - k(z,w) + k(0,w) \nn \\
& =& (b(z) -b(0)) b(w)^* \quad \forall z,w \in \B^d. \nn \ea
If $zB(z) - (b(z) -b(0)) \neq 0$ then there is a non-zero $h$, $$ h \in \left( \bigvee _{w \in \B^d} \ran{b(w)}^* \right) ^\perp = \bigcap _{w \in \B^d} \ker{b(w)}, $$ so that $$ 0 \neq (z B(z) - (b(z) - b(0)) ) h = zB(z) h. $$ However $B = U_A^* E^* m_b \delta _0$ so that $z B(z) h = k_z ^* z Bh$. We will show that $h$ is in the kernel of $m_b \delta _0$. Recall that $m_b : \H (\delta ) \rightarrow \H (k^F)$ is a co-isometry. Since $m_b ^* k^F _z = \delta _0 b(z) ^*,$ the initial space of $m_b$ is
$$ \ker{m_b} ^\perp = \bigvee _{z \in \B ^d } \delta _0 b(z) ^* \H. $$ If $h \in \cap _{w\in \B ^d} \ker{b(w)}$ then
\ba \ip{ \delta _0 h}{\delta _0 b(z) ^* g} & = & \ip{h}{\delta (0,0) b(z) ^* g} \nn \\
& = & \ip{h}{(I - b(0) ^* b(0)) b(z)^* g } =0, \nn \ea since $$ (I-b(0) ^* b(0)) b(z) ^* g \in \bigvee _{w \in \B ^d} b(w) ^* \H \perp h. $$
This proves that $\delta _0 h \in \ker{m_b}$ so that $z B (z) h =0$ and
$$ z B(z) = b(z) - b(0), \quad \forall z \in \B^d. $$
\end{proof}

\begin{remark}
    One can develop an alternative proof of the above theorem, at least in the scalar case, using the Douglas factorization lemma and a maximum modulus principle argument.
\end{remark}

\begin{lemma} \label{Glemap}
Let $D \supseteq V$ be a $d-$contractive extension of $V$ on $\L(b)$. Then
$$ B ^D := U^* D  ^* U k_0 ( I - b(0) ^* ) ^{-1} (I - b(0) )  = U^* D^* K_0 (I -b(0)), $$ defines a contractive Gleason solution for $b$. If
$b$ is quasi-extreme then $B^V$ and $X^V := X(B^V)$ are extremal.
\end{lemma}
\begin{proof}
Since $D$ is a $d-$contractive extension of $V$, 
$$ Dz^* K_z  = K_z - K_0. $$

Consider
\ba K_z^* zD^* K_0 & = & K(z,0) - K(0,0) \nn \\
& = & (I-b(z) ) ^{-1} (b(z) - b(0)) (I-b(0) )^{-1}.  \nn \ea
This proves that
$$ (z D^* K_0 ) (z)= \sum z_j (D_j ^* K_0) (z) =  (1-b(z)) ^{-1} (b(z) - b(0) ) (1-b(0) ) ^{-1} $$
Solving for $(b(z) - b(0) )$ in the above equation then yields
$$ b(z) - b(0) = z \left(  (I-b(z)) (D ^* K_0 ) (z) (I-b(0) ) \right). $$
It follows that
$$ B (z) := (I-b(z)) (D ^* K_0 ) (z) (I-b(0) ), $$ defines a solution to the Gleason problem. Alternatively,
using the canonical unitary multiplier $U = U_b : K (b) \rightarrow \L (b)$ this can be written as
$$ B (z) = (U^* D ^* U  k_0) (z) (I-b(0 ) ^* ) ^{-1} (I-b(0)).$$

Since $D$ is a $d$-contraction it follows that
\ba B^* B & = & (I-b(0) ^* ) K_0 ^*  D  D ^* K_0  (I-b(0)) \nn \\
& \leq & (I-b(0) ^*) K (0,0) (I-b(0)). \label{Gextremal} \\
& = & \frac12 (I-b(0) ^* ) (I+b(0)) + (I + b(0) ^* ) (I-b(0)) \nn \\
& = &  (I-b(0) ^* b(0) ), \nn \ea and $B$ is a contractive Gleason solution for $b$.

By Theorem \ref{Glechar}, if $B^V$ is extremal, so is $X^V = X(B^V)$. If $b$ is quasi-extreme, $V=V^b$ is a co-isometry so that
equality holds in the second line, (\ref{Gextremal}), of the above equation. This proves that $B=B^V$ and hence $X^V$ are extremal if $b$ is quasi-extreme.
\end{proof}
Let $X ^D = X(B^D)$ denote the contractive Gleason solution for $K(b)$ constructed using the contractive extension $D \supseteq V^b$ as in the previous lemma.

\begin{thm} \label{qeuniqsol}
The map $D \mapsto B^D$ is a bijection from contractive Gleason solutions for $\L (b)$ onto contractive Gleason solutions for $b$.
\end{thm}
\begin{proof}
    The previous lemma shows that given any contractive Gleason solution $D \subseteq V = V^b$ for $\L (b)$ that
$$ B^D := U^* D^* K_0 (I-b(0)), $$ is a contractive Gleason solution for $b$. This map $D \mapsto B^D$ is clearly injective since
if $B = B^D $ and $B' = B^C$ for $C, D \supseteq V$ then
$$ D^* K_0 = C^* K_0. $$ Since $C^*, D^*$ are both extensions of $V^*$, we have
$$ D^* (K_z - K_0) = V^* (K_z - K_0) = C^* (K_z - K_0), $$ so that $D^* K_z = C^* K_z$ for all $z \in \B^d$ and $D=C$.

To prove that this map is surjective, we will show that its inverse can be defined on the set of all contractive Gleason solutions for $b$.
If $B$ is an arbitrary contractive Gleason solution for $b$, define a bounded linear map $(D^B) ^*  =: D^* : \L (b) \rightarrow \L (b) \otimes \C ^d$
by $$ D^* K_z := z^* K_z + UB(I-b(0)) ^{-1}; \quad \quad z \in \B ^d.$$ This construction is based on the observation that if $C \supseteq V$ and  $B = B^C$ then
\ba C^* K_z & = & z^* K_z + C^* K_0 \nn \\
& = & z^* K_z + U B^C (I - b(0))^{-1}. \nn \ea
By construction,
$$ D^* (K_z - K_0 ) = z^* K_z + UB (I - b(0)) - UB (I-b(0)) = z^* K_z, $$ and so if we can prove that $D^*$ is a contraction, the fact that
$D^* \supseteq V^*$ will imply that $D \supseteq V$ is a contractive Gleason solution for $\L (b)$ by Lemma \ref{contractext}.
Calculate the norm of $D^*$ on kernel maps:
\ba K_z ^* D D^* K_w & = & \left( K_z ^* z + (I-b(0)^*) ^{-1} B^* U^* \right) \left( w^* K_w + UB (I-b(0)) ^{-1} \right) \nn \\
& = & zw^* K(z,w) + (I-b(0) ^* ) ^{-1} B^* w^* k_w (I-b(w) ^*) ^{-1}  \nn \\
& &  + (I-b(z) ) ^{-1} k_z ^* z B (I-b(0)) ^{-1} + (I-b(0) ^* ) ^{-1} B^* B (I-b(0)) ^{-1} \nn \\
& \leq & zw^* K(z,w) + (I -b(0) ^* ) ^{-1} (b(w) ^* - b(0) ^*) (I-b(0)^* ) ^{-1} \nn \\
& & + (I-b(z) ) ^{-1} (b(z) - b(0) ) (I -b(0) ) ^{-1}
+ ( I-b(0) ^*) ^{-1} (I-b(0) ^* b(0) ) ( I - b(0) ) ^{-1} \nn \\
& = & zw^* K(z,w) + (K(0,w) - K (0,0) ) + (K (z,0) - K(0,0)) + K(0,0) \nn \\
& = & K(z,w) -\frac{1}{2} (H(z) + H(w) ^*) + K(0,w) + K(z,0) - K(0,0) \nn \\
& = & K(z,w). \nn \ea  This proves that $D^*$ is a contraction so that $D \supseteq V$ is a contractive Gleason solution for $\L (b)$.

This map $B \mapsto D^B$ is injective since if $C = D^{B_1} = D = D^{B_2}$ for contractive Gleason solutions $B_1, B_2$ for $b$ then necessarily
$$ B_1 (I-b(0) ) ^{-1} = B_2 (I -b(0) ) ^{-1}, $$ by the definition of $D^B$. Moreover if $D ' := D^{B^D}$ then
\ba (D' ) ^* K_z & =& z^* K_z + U B^D (I -b(0) ) ^{-1} \nn \\
& = & z^* K_z + D^* K_0 (I -b(0)) (I- b(0)) ^{-1} \nn \\
& = & D^* (K_z -K_0) + D^* K_0 \nn  \\
& = & D^* K_z, \ea so that $D' = D$. It follows that the maps $D \mapsto B^D$ and $B \mapsto D^B$ are inverses to one another and define bijections.
\end{proof}

\begin{cor} \label{min}
    The contractive Gleason solution $B^V$ for $b$ is minimal and unique:
     $$ (B^V) ^* B^V   \leq B^* B$$  for all contractive Gleason solutions, $B$, for $b$. Equality holds if and only if $B = B^V$.
     Similarly $(X^V) (X^V)^* \leq X X ^*$ for all contractive Gleason solutions, $X$, for $K(b)$, with equality holding if and only if $X = X^V$.
\end{cor}

\begin{proof}
Given any $d-$contractive extension $D$ of $V$, $D = V + C$ where $C : \ker{V} \rightarrow \ran{V} ^\perp$ is a $d$-contraction so that $D D ^* = V V ^* + C C ^*$. By Theorem \ref{qeuniqsol}, the map $D \mapsto B^D$ is onto, so that we can assume $B = B^D$ for such a contractive extension $D$. Recall that
$$ B^D = U^* D ^* A^*, $$ where $A^* \in \L ( \H , \L (b) )$ is defined by
$$ A^* = U k_0 (I-b(0) ^* ) ^{-1} (I-b(0)) = K_0 (I-b(0)). $$ It follows that
\ba (B^D) ^* B^D & = & A D D ^* A^*  \nn \\
& = & A VV ^* A^* + A C C ^* A ^* \nn \\
& \geq  & A V V ^* A^*  \nn \\
& = & (B^V )^* B^V. \nn \ea
Moreover $(B^D) ^* B^D = (B^V) ^* B^V$ if and only if $C^* A^* =0$. Since $\ran{A^*} = \bigvee K_0 \H$ and $\ran{C} \subset \ran{V} ^\perp$ is spanned by constant functions, it follows that $C^* A^* =0$ if and only if $C = 0$, which happens if and only if $D=V$ and $B^D = B^V$.

Similarly, and without loss of generality, we can assume that $X = X^D = X(B^D)$ for some contractive extension $D$ of $V$ on $\L (b)$. Then,
\ba k_z ^* X ^D (X^D)^* k_z & = & (z^* k_z - B^D b(z)^* ) ^* (z^* k_z - B^D b(z)^* ) \nn \\
&= & z z^* k(z,z) -(b(z) - b(0) ) b(z) ^* - b(z) (b(z)^* - b(0) ^* ) + b(z) (B^D) ^* B^D b(z) ^*, \ea  so that
$$ k_z ^* \left( X^D  (X^D) ^* - X^V (X^V) ^* \right) k_z = b(z) \left( (B^D)^* B^D  - (B^V) ^* B^V \right) b(z) ^* \geq 0.$$
Equality holds if and only if
$$ B^V  b(z)^* = B^D b(z) ^* \quad \forall z  \in \B^d,$$ in which case $X^V = X(B^V) = X(B^D) = X^D = X$.
\end{proof}

\begin{remark} \label{minremark}
    It follows that if $b$ is not quasi-extreme, then $B^V$ is not extremal so that
    $(B ^V) ^* B^V < I - b(0) ^* b(0)$.
\end{remark}

\begin{thm} \label{Glebij}
Suppose that $d>1$. The map $B \mapsto X(B)$ is a bijection if and only if $\bigvee _{z \in \B ^d } b(z)^* \H = \H$ or equivalently $\bigcap _{z\in \B^d } \ker{b(z)} = \{ 0 \}$.
If this condition holds then the contractive Gleason solution $X(B)$ is extremal if
and only if $B$ is extremal. In particular this holds if $b$ is quasi-extreme.
\end{thm}
The condition $\bigcap _{z \in \B ^d} \ker{b(z)^*} = \{ 0 \}$ says that $b$ has no identically `zero columns', and this always holds in the scalar-valued case where $b \in [H^\infty _d ]_1$. One usually defines $\mr{supp} (b) := \bigvee _{z \in \B ^d} \ran{b(z)^*}$, as in \cite{BES2006cnc}.
\begin{proof}
By Theorem \ref{Glechar}, given any contractive Gleason solution $X$ for $K(b)$, there is a contractive Gleason solution $B$ for $b$ so that $X = X(B)$, and the map $B \mapsto X(B)$ from contractive Gleason solutions for $B$ to contractive Gleason solutions for $K(b)$ is onto.

Suppose that $\bigvee _{z \in \B ^d} \ran{b(z)^*} = \H.$ If $X = X(B) = X(B')$ where $B, B'$ are two contractive Gleason solutions for $b$. Then for all $z \in \B^d$, $$ B b(z)^* = B' b(z) ^*, $$ and the assumption on the closed span of the ranges of the $b(z) ^*$ implies that $B = B'$ so that the map $B \mapsto X(B)$ is a bijection.

Recall the notation of the proof of Theorem \ref{Glechar}.  The assumption that $\bigvee _{z \in \B ^d} \ran{b(z)^*} = \H$ implies that $m_b : \H (\delta ) \rightarrow \H (k^F )$ (as defined in the proof of Theorem \ref{Glechar}) is an onto isometric multiplier: Since $\delta (z,w) = (I -b(0) ^* b(0))$ is a positive constant operator, any function $G \in \H (\delta )$ is a constant, $G(z) = G(0) = g \in \H$. Moreover $G \in \ker{m_b}$ if and only if $$ b(z) G(z) = b(z) g =0; \quad \quad \forall z \in \B^d.$$ By assumption $\cap _{z \in \B^d} \ker{b(z)} = \{0 \}$ so that $G \equiv 0$ and $m_b ^* m_b = I _{\H (\delta )}$.

If $X$ is extremal and $b$ obeys this range condition then $k^F = k^A$ and $A_w = U_A ^* k^F _w$ so that $B = U_A ^* m_b \delta _0$ and
$$ B^* B = \delta _0 ^* m_b^* m_b \delta _0 = \delta (0,0 ) = I - b(0) ^* b(0). $$ This proves that $B$ is extremal.

If $\bigvee _{z \in \B ^d } b(z) ^* \H \neq \H$ then its orthogonal complement $\bigcap _{z \in \B ^d} \ker{b(z)} $ is non-trivial, so that there is a non-zero $h \in \H$ such that $b(z) h =0$ for all $z \in \B^d$,
and $$ b h = 0 \in K(b). $$ By Theorem \ref{containb}, $b$ is not quasi-extreme in this case.

Now suppose that there is a non-zero $h$ orthogonal to $\bigvee b(z) ^* h$. Recall by Remark \ref{ntkernremark} since $d>1$, the kernel of $V = V^b$ is non-trivial so that
$$ U_b ^* \ker{V} ^\perp = \bigvee z^* k_z \H, $$ is not all of $K(b) \otimes \C ^d$. Recall by Lemma \ref{Glemap} that $B := B^V$ is given by the formula
$$ B = U_b ^* V^* K_0 (I - b(0) ), $$ and so by Remark \ref{minremark} above, we can construct a rank-one map $B' $ from the one-dimensional span of $h$ into the orthogonal complement
of $\bigvee z^* k_z \H $  such that
$$ (B + B' ) ^* (B + B') = B B^* + B' (B' ) ^* \leq I - b(0) ^* b(0). $$ Here we used that $B'$ takes $h$ into the orthogonal complement of the range of $B = B^V$. Then $B + B'$ is a contractive
Gleason solution for $b$ since
$$ z B' (z) = (z^* k_z ) ^* B' = 0, $$ so that
$z (B + B' ) (z) = z B(z) = b(z) - b(0).$ Finally,
\ba X (B + B' ) k_z & = & z^* k_z - (B + B' ) b(z) ^* \nn \\
& = & z^* k_z - B b(z) ^* = X(B). \nn \ea This proves that the map $B \mapsto X(B)$ is not injective.
\end{proof}

\begin{cor}
The map $D \mapsto X^D := X(B^D )$ is a surjection from contractive Gleason solutions $D \subseteq V$ for $\L (b)$ onto contractive
Gleason solutions for $K(b)$.  This map is injective if and only if $\bigcap _{z\in \B ^d} \ker{b(z)} = \{ 0 \}$.
\end{cor}

\begin{cor}
A Schur class $b \in [H^\infty _d \otimes \L (\H ) ]_1$ is quasi-extreme if and only if it has a unique contractive and extremal Gleason solution $B = B^V$. If $b$ is quasi-extreme then $K(b)$ has a unique contractive extremal Gleason solution $X = X^V$. If $K(b)$ has a unique contractive Gleason solution $X$ and $\bigcap _{z \in \B^d} \ker{b(z)} = \{ 0 \}$, then $b$ is quasi-extreme.
\end{cor}

\begin{remark}
    Note that it can happen that $b$ is not quasi-extreme but $K(b)$ has a unique contractive Gleason solution if $\bigvee \ran{b(z)^*} \neq \H$.
For example, if $b \in [H^\infty _d] _1$ is not quasi-extreme then one can construct an $a \in [H^\infty _d ]_1$ that is a several variable analogue
of the outer function with modulus $\sqrt{1 - |b| ^2 } $ on the boundary of the circle in the case where $d=1$ \cite[IV-1]{Sarason-dB}. In this case one can
prove that
$$ c := \begin{bmatrix} b & 0 \\ a & 0 \end{bmatrix} \in [H^\infty _d \otimes \C ^{2\times 2} ] _1, $$ and that any contractive Gleason solution $X$ for $K(c)$
is extremal so that $K(c)$ has a unique contractive Gleason solution by Corollary \ref{min} above. In this example $\bigcap _{z \in \B ^d } \ker{c(z)} = \{ e_2 \}$,
where $\{ e_1, e_2 \}$ is the canonical orthonormal basis of $\C ^2$, so that $ c e _2 = 0 \in K (c)$ and $c$ is not quasi-extreme.

The construction of this `outer' function $a$ will be presented in a future publication where we analyze the convex (and matrix convex) structure of $[H^\infty _d \otimes \L (\H ) ]_1$,
and the relationship between quasi-extreme points and extreme points of this convex set.
\end{remark}

Using Gleason solutions, we can prove in most cases that $V^b$ has non-trivial kernel when $d>1$, a fact we used in Section \ref{Herglotz}, see Remark \ref{ntkernremark}.

\begin{prop} \label{ntkernel}
    If $b \in [H^\infty _d \otimes \L (\H )] _1$ where $d>1$ has a contractive Gleason solution $B = (b_1 , ... , b_d )^T$ such that
    $$ \bigvee _{1\leq j \leq d} b_j \H \neq K(b), $$ then $V^b$ has non-trivial kernel.
\end{prop}
\begin{proof}
Let $X = X^B$ be a contractive Gleason solution for $K (b)$. Then $B \in \L ( \H , K(b) \otimes \C ^d)$. Choose any $F \in K(b)$ orthogonal to the range of the $b_j \in \L (\H , \K (b))$ where
$B = \begin{bmatrix} b_1 \\ \vdots \\ b_d \end{bmatrix}.$
It follows that
\ba \ip{ (X_j F) }{ k_z h} & = & \ip{F}{\ov{z} _j k_z h - b_j b(z) ^* h } \nn \\
& = & \ip{z_j F (z)}{h} _\H. \nn \ea This proves that for any $1 \leq j \leq d$, $ S_j F  = G_j \in K(b)$ where $G _j (z) = z_j F (z)$. This in turn implies that $H_j := (I-b) ^{-1} G_j \in \L (b)$
and if we define $\mbf{H} := (-H_2 , H_1 , 0 , ..., 0 ) \in \L (b) \otimes \C ^d$, we then have that
\ba \ip{\mbf{H}}{z^* K_z h } & = & - z_1 \ip{H_2 (z)}{h} +z_2 \ip{H_1 (z)}{h} \nn \\
& = & (-z_1 z_2 + z_2 z_1 ) (I-b(z) ) ^{-1} F(z)  = 0. \nn \ea
It follows that $\ker{V^b}$ is non-trivial.
\end{proof}

\subsection{Clark's perturbations and Intertwining}
In this section we verify the intertwining formulas for perturbations of the minimal contractive solution $X := X^V$ to the Gleason problem in $K(b)$ and the compression of $R := \pi _{\nu _b} (L) = \pi _b (L) $ to $P^2 (b)$.
Here recall that $\nu _b$ is the tight extension of $\mu _b$ and if $P$ is the projection onto $P^2 (b )=P^2 (\mu _b )$ then $\mc{C} _b ^* V ^b \mc{C} _b = \hat{V} = P R P = P R$ since $P^2 ( b )$ is co-invariant for $R$.

Let $\mc{F} _b : P^2 (b ) \rightarrow K (b)$ be the weighted Cauchy transform.  Recall that $\mc{C} _b = U_b ^* \mc{F} _b $ where  $U_b : K(b) \rightarrow \L (b)$ is the canonical onto isometric multiplier. Also recall that $\mc{C} _b R = W \mc{C} _b$ where $(W , \K _b )$ is the minimal $d-$isometric dilation of $(V , \L (b) )$.

We will obtain the desired intertwining formulas by first calculating intertwining formulas for $V$ and $X$ via the unitary $U = U_b$, where $X:= X^V$ is the minimal Gleason solution corresponding to $V=V^b$. We have that
$$ X^* k_w = w^* k_w - B b(w) ^*, $$ where $$ B:= U^* V^*  K_0 (I-b(0)). $$ 
Compare this to
\ba U^* V ^* U k_w & = &  U^* V ^*  K_w (I - b(w) ^* ) \nn \\
& = &  U^* V ^* (K_w -K_0 ) (I - b(w) ^* ) +  U^* V ^* K_0 ( I- b(w) ^* ) \nn \\
& = & w^* k_w + U^* V^* K_0 (I-b(w) ^*) \nn \\
&=& w^* k_w + B (I-b(0) )^{-1} (I-b(w) ^* ). \nn \ea

Define $ T : K(b) \rightarrow K(b) \otimes \C ^d$ by
$$ T := B (I- b(0)) ^{-1} k_0 ^*. $$ Then
\ba T k_w & = & B (I-b(0) )^{-1} k_0 ^* k_w \nn \\
& = & B (I- b(0) ) ^{-1} k(0, w)  \nn \\
& = & B (I-b(0)) ^{-1} (I - b(0) b(w) ^* ) \nn \\
& = & B (I-b(0))^{-1} \left( I - b(w)^* + b(w)^* - b(0) b(w)^* \right) \nn \\
& = & B \left( (I -b(0) ) ^{-1} (I -b(w)^*) + b(w)^* \right). \nn \ea
It follows that
\ba (X ^* + T) k_w & = & w^* k_w  - B b(w) ^* + B b(w) ^* + B (I-b(0)) ^{-1} (I-b(w) ^* ) \nn \\
& = & U^* V ^* U k_w. \nn \ea
This proves the intertwining formula:
\ba U^* V ^* & = & (X^* +T) U^* \nn \\
& = & \left( X ^* + B (I-b(0) ) ^{-1} k_0 ^*  \right) U^*, \ea or equivalently,
\be \mc{F} _b \hat{V}  ^* = \left( X_j ^* + B (I-b(0) ) ^{-1} k_0 ^*  \right) \mc{F} _b, \ee
where $\hat{V} = P \pi _b (L) P = \mc{C} _b ^* V \mc{C} _b$ and $P : Q^2 (b) \rightarrow P^2 (b)$ is orthogonal projection.

Recall that for any unitary $A$ on $\H$ we define the Aleksandrov-Clark CP map $\mu _A := \mu _{b \cdot A^*}$. It is clear that $K(b\cdot A^* ) = K(b)$.
Let $\hat{V} ^A$ be the compression of $\pi _{\nu _A} (L) =: R^A$ to $P^2 (\mu _A)$. Recall $\hat{V} ^A$ is a partial isometry which is unitarily equivalent via the Cauchy transform to $V ^{bA^*}$. Here $\nu _A$ is the tight extension of $\mu_A$. Let $V^A = V ^{bA^*}$ be the partial $d-$isometry on $\L (b\cdot A^* )$, $X ^A$ be the minimal contractive Gleason solution in $K(b)$ corresponding to $V^A$, and  $B^A$ be the corresponding Gleason solution for $b$. We
will simply write $X = X^I$ and $B = B^I$.

\begin{thm}{ (Partial isometric Clark perturbations)} \label{Clark}
For any $b \in [H^\infty _d \otimes \L (\H ) ] _1$ and unitary $A \in \L (\H)$,
\be \mc{F} _{bA^*} (\hat{V}  ^A) ^* = \left( X ^* + B A^*(I-b(0)A^* ) ^{-1} k_0 ^*  \right) \mc{F} _{bA^*}. \ee
\end{thm}
See \cite[Theorem 5.1]{Jur2014AC} for the case where $b \in [H^\infty _d ] _1$ is quasi-extreme.
\begin{proof}
As discussed above, let $B^{A} = B ^{V ^A}$ denote the Gleason solution in $K (bA^*) =K(b)$ corresponding to the partial isometry $V^A = V^{bA^*}$ in $\L (bA^*)$.
By Corollary \ref{min}, for any unitary $A \in \L (\H )$, the Gleason solution $B^A$ is the unique minimal contractive Gleason solution for $bA^*$, \emph{i.e.} $(B^A) ^* B^A \leq (\hat{B} ^A) ^* \hat{B} ^A$
for any other contractive Gleason solution $\hat{B} ^A$ for $bA^*$. Uniqueness implies that $B^A \cdot A = B^I =B$ for any unitary $A \in \L (\H )$ where $B$ is the minimal contractive Gleason solution for $b$. It follows that for any $w \in \B ^d$,
\ba  (X^A) ^* k_w ^b & = & w^* k_w ^b - B^A (b(w) A^* )^* \nn \\
& = & w^* k_w ^b - B b(w) ^* \nn \\
& = & (X^I ) ^* k_w ^b, \nn \ea  so that $X ^A = X^I = X$, the minimal contractive Gleason solution for $K(b)$.  Repeating the arguments preceding the statement of this theorem then yields the desired intertwining formulas.
\end{proof}

Also recall that if $b$ is quasi-extreme then $V$ is a co-isometry so that each of the $\mc{F} _{bA^*} \hat{V} ^A \mc{F} _{bA^*} ^* = U_b ^* V^A U_b$ are co-isometric
perturbations of the (unique) contractive Gleason solution $X$ for $K(b)$.

The Clark-type intertwining formulas can also be used to provide a simple proof that the kernel of $V^b$ is non-trivial in the case where $d>1$ and $\H$ is finite dimensional:

\begin{prop} \label{ntkernel2}
    If $b \in [H^\infty _d \otimes \L (\H )]_1$, $d>1$ and $\dim{\H } < \infty$ then the kernel of $V^b$ is non-trivial.
\end{prop}

\begin{proof}
    If $\H$ is finite dimensional then the formula of Theorem \ref{Glechar} implies that if $X = X^V$ is the contractive Gleason solution corresponding to $V = V^b$ then
the commutators $[ X_j , X_k ]$, $1 \leq j,k \leq d$ all have finite rank. By the intertwining formula of Theorem \ref{Clark} above it follows that the commutators
$[V_j , V_k]$ also all have finite rank. If $V$ had trivial kernel then $I - V^* V = 0 $ and $I -V V^*$ both have finite rank (since we assume $\H$ is finite dimensional).

Taking the quotient by the compact operators, the image of $V$ is a commutative $d$-contraction obeying the Cuntz relations.  This is impossible and proves the statement.
\end{proof}

\subsection{Summary}

The previous sections have developed several equivalent characterizations of the quasi-extreme Szeg\"{o} approximation property. These are summarized as follows:

\begin{thm} \label{summarythm}
    Let $b \in [H^\infty _d \otimes \L (\H ) ] _1$. The following are equivalent:
\bn
    \item $b$ is quasi-extreme, \emph{i.e.} $P^2 (b) = P^2 _0 (b)$.
    \item The partial $d-$isometry $V^b$ on $\L (b)$ is a co-isometry; equivalently $\L (b)$ contains no constant functions.
    \item The row-isometry $\pi _{b} (L)$ on $Q^2 (b)$ is a Cuntz unitary.
    \item The CP map $\mu_b:\mc{S}+\mc{S}^*\to B(\H)$ has a unique CP extension to $\mc{A} + \mc{A} ^*$.
    \item There is no non-zero $h \in \H$ so that $bh \in K(b)$.
    \item There is a unique contractive solution to the Gleason problem for $b$, and this solution is extremal.
    \item There is a unique contractive solution to the Gleason problem in $K(b)$ and \\
    $\bigcap _{z \in \B ^d} \ker{b(z)} = \{ 0 \}$. Such a solution is extremal.
\en
\end{thm}

\section{Examples: Inner Sequences}\label{Examplessection}

In this section we provide some examples in the case where $b$ is a matrix-valued multiplier associated to an inner sequence (defined below).

\begin{defn}
We will use the notation $H^\infty _d \otimes \C ^{n\times m} := \mr{Mult} ( H^2 _d \otimes \C ^m , H^2 _d \otimes \C ^n )$ for the multipliers from $H^2 _d \otimes \C ^m $ into $H^2 _d \otimes \C ^n$.
We write $H^\infty _d \otimes \C ^m$ for $H^\infty _d \otimes \C ^{1 \times m}$.

As is well known, a linear map $M \in \L ( H^2 _d \otimes \C ^n , H^2 _d \otimes \C ^m )$ is a multiplier,  if and only if $M  (M_\varphi \otimes I _{n} ) = (M _\varphi \otimes I _m ) M $ for all $\varphi \in H^\infty _d$ \cite{Sha2013}. In this case $M = M_\Theta$ for some $\Theta \in H^\infty _d \otimes \C ^{n\times m}$.
\end{defn}

Recall that a multiplier $\Theta \in H^\infty _d \otimes \C ^{n \times n }$ is called \emph{inner} if $M _\Theta$ is a partial isometry. Any shift invariant subspace $\mathcal{M}\subset H^2 _d$ is the range of an inner $\theta \in H^\infty _d \otimes \C ^n$ for some $n \in \N \cup \{\infty \}$ \cite{McTrent2000,Arv1998curv,GRS2002}. We embed $H^\infty _d \otimes \C ^n$ in $H^\infty _d \otimes \C ^\infty = H^\infty _d \otimes \H$ in the natural way (add zeroes). We can further embed $H ^\infty _d \otimes \C ^\infty$ into $H ^\infty _d \otimes \C ^{\infty \times \infty} = H^\infty _d \otimes \L (\H )$ via the map
$$\theta (z) := (\theta _1 (z) , \theta _2 (z) , ... ) \mapsto \begin{bmatrix} \theta _1 (z)  & \theta _2 (z) & \cdots & \quad \quad \\ 0 & 0 & & \\ \vdots & & \ddots & \\ 0 & & &  \quad \quad \end{bmatrix} =: \hat{\theta} (z). $$

$( \theta _k ) = \left( \theta _1  , \theta _2  , ... \right)$ is called an \emph{inner sequence} associated to $\mathcal M$.
It follows that if $\theta \in H^\infty _d \otimes \C ^n$ is inner then $\Theta := \hat{\theta} \in [H^\infty _d \otimes \C ^{n\times n} ] _1 $ is inner and the reproducing kernel for $K(\Theta )$ is
$$ k ^\Theta = k ^\theta \oplus \left( k \otimes I _{n-1 \times n-1} \right), $$ where $k$ is the Szeg\"{o} kernel. The map $\theta \mapsto \Theta =\hat{\theta }$ is a completely isometric embedding of $H^\infty \otimes \C ^n$ into $H^\infty \otimes \C ^{n\times n}$.

\begin{prop} \label{innerqe}
    Suppose that $\Theta \in [H^\infty _d \otimes \L (\H ) ]_1$ is inner. Then $\Theta$ is quasi-extreme if and only if $\ker{M_\Theta} \cap \bigvee k_0 \H = \{ 0 \}$.
\end{prop}

\begin{proof}
   By Theorem \ref{containb}, $\Theta $ is not quasi-extreme if and only if there is a non-zero $h \in \H$ so that $\Theta h \in K(\Theta )$. However, since
$\Theta $ is inner, $K(\Theta ) = \ran{M_\Theta } ^\perp$ so that $\Theta h \in K (\Theta )$ if and only if $k_0 h \in \ker{ M_\Theta }$.
\end{proof}

The condition for quasi-extremity just given can be recast in a more elegant form when $\Theta =\hat{\theta}$ comes from an inner sequence $ \theta = (\theta_j) \in H ^\infty _d \otimes \C ^n$. If $(\theta_j)_{j\in J}$ is an inner sequence for some (finite or countable) index set $J$, observe that for each $z\in\mathbb B^d$ we have $(\theta_j(z))\in \ell^2(J)$. Say that $\theta$ is {\em minimal} if
  \begin{equation*}
    \bigvee _{z \in \B ^d} (\theta_j(z)) _{j \in J} = \ell ^2 (J).
  \end{equation*}

\begin{prop} \label{nqe}
    Let $\theta \in H^\infty _d \otimes \C ^n$, $n \in \N \cup \{\infty \}$, $\theta = (\theta _1 , \theta _2 , ... , \theta _n )$ be an inner sequence.
Then the embedding $\Theta = \hat{\theta} \in H^\infty _d \otimes \C ^{n\times n }$ is quasi-extreme if and only if $\theta$ is a minimal inner sequence.
\end{prop}

\begin{proof}
By the previous corollary $\Theta$ is quasi-extreme if and only if $\ker{M_\theta } \cap \bigvee k_0 \C ^n = \{ 0 \}$. The claim follows easily from this fact.
\end{proof}

\begin{remark} \label{genqe}
We observe that a finite inner sequence $\theta =(\theta_1, \dots \theta_n)$ is minimal if and only if the functions $\theta_1, \dots \theta_n$ are linearly independent (in the space of holomorphic functions on $\mathbb B^d$). Now, it is known \cite{McTrent2000,Arv1998curv,GRS2002} that every closed $S$-invariant subspace $\mathcal M\subset H^2_d$ is represented by an inner sequence, in the sense that the multiplication operator $M_\theta$ where  $\theta = (\theta_j)_{j\in J}$ is a partially isometric multiplier of $H^2_d\otimes \ell^2(J)$ onto $\mathcal M$. It is not difficult to see that $\theta$ may always be chosen minimal: Indeed, if $\theta$ is any such multiplier, let $\mathcal H$ be the closed span of $\{(\theta_j(z)):z\in\mathbb B^d\}$ in $\ell ^2 (J)$. Identify $\mathcal H$ with $\ell^2(K)$ for an appropriate index set $K$; define $\psi(z)$ to be the compression of $\theta (z)$ to $\mathcal H$. Expanding $\psi (z)$ in an orthonormal basis for $\H = \ell^2(K)$ gives a minimal inner sequence $\psi = (\psi_k)_{k\in K}$ which multiplies $H^2_d\otimes \ell^2(K)$ onto $\mathcal M$.
\end{remark}

\begin{eg}\label{eg:inner-sequence} Let $d=2$ and let $\mathcal M$ be the orthogonal complement of the span of $\{1, z_1, z_1^2, z_2\}$ in $H^2_2$. One may verify that the $4$-tuple
  \begin{equation}\label{eqn:eg-inner-sequence}
    (z_1^3, z_1^2z_2,\sqrt{2}z_1z_2, z_2^2),
  \end{equation}
is an inner sequence representing $\mathcal M$, and since the monomials are linearly independent this sequence is minimal, and hence the $4\times 4$ matrix function
\begin{equation} \label{squareinner}
  \Theta(z_1, z_2) = \begin{pmatrix}z_1^3 &  z_1^2z_2 & \sqrt{2}z_1z_2 &  z_2^2 \\
0 & 0 & 0 & 0 \\
0 & 0 & 0 & 0 \\
0 & 0 & 0 & 0 \end{pmatrix}
\end{equation}
is a quasi-extreme inner multiplier.
\end{eg}

If $b \in [H^\infty _d \otimes \L (\H )]_1$ is inner then by Proposition \ref{innerqe}, Proposition \ref{nqe} and Remark \ref{genqe} it is generally quasi-extreme. By Theorem \ref{qeuniqsol}, there is a unique solution to
the Gleason problem in the co-invariant model space $K(b)$. This solution is given by the compressed shift, $S_b = P _b S | _{K(b)}$ where $P_b$ projects onto $K(b)$. Indeed, using that $P_b$ is co-invariant,
\ba S_b z^* k_z ^b h & = & (z^* S_b) P_b k_z h \nn \\
& = & (z^* S _b) P_b (I-z^* S) ^{-1} k_0 ^b h \nn \\
& = & (z^* S_b) (P_b -z^* S_b ) ^{-1} k_0 ^b h  \nn \\
& = &  (P_b - z^* S_b ) ^{-1} k_0 ^b h - k_0 ^b h \nn \\
& = & k_z ^b h  - k_0 ^b h. \nn \ea

A natural question is the following: what are the contractive Gleason solution components $b_j$ associated to the unique solution $S_b$? A natural conjecture
is $$ b_j (z) h := (S_j ^* b) (z) h. $$ For this to be the case it is necessary that  $S_j ^* b h \in K(b)$ for all $h\in\mathcal H$.  However, when $d>1$, this is not always so; in particular it fails for the example just considered above in Equation (\ref{squareinner}, $b = \Theta$.  Let $\theta \in H^\infty _d \otimes \C ^4$ be the first row of $b = \Theta$, an inner multiplier.
Recall that the components $S_j^*$ of the backward Arveson $d-$shift act on monomials by
\begin{equation*}
  S_j^* z^{\alpha} =\frac{\alpha_j}{|\alpha|} z^{\alpha-e_j} \quad \text{if } \alpha_j\geq 1
\end{equation*}
and $S_j^*z^\alpha=0$ otherwise \cite{Sha2013,ArvIII}. For this $b$ we have
\begin{equation*}
  (S_1^*b)(z_1, z_2) =  \begin{pmatrix} z_1^2 & \frac 23 z_1z_2 & \frac{1}{\sqrt{2}} z_2 & 0 \\
0 & 0 & 0 & 0 \\
0 & 0 & 0 & 0 \\
0 & 0 & 0 & 0 \end{pmatrix}
\end{equation*}
But $K(b)= [\ran{M_\theta} ] ^\perp  \oplus H^2_d \otimes \C ^3$, while for the basis vector  $e_2\in\mathbb C^4$ we have $(S_j^*b)e_2 = (\frac 23 z_1z_2, 0, 0, 0 )^T$ ($T$ denotes transpose) which does not lie in $K(b)$ (since $z_1z_2$ lies in $\ran{M_\theta}$).  It is possible to compute that $B_1$ and $B_2$ are in this case given by
\begin{equation*}
  B_1 = \begin{pmatrix} z_1^2 & 0 & \frac{1}{\sqrt{2}}z_2 & 0 \\
0 & 0 & 0 & 0 \\
0 & 0 & 0 & 0 \\
0 & 0 & 0 & 0 \end{pmatrix}, \quad B_2 = \begin{pmatrix} 0 & z_1^2  & \frac{1}{\sqrt{2}}z_1 & z_2 \\
0 & 0 & 0 & 0 \\
0 & 0 & 0 & 0 \\
0 & 0 & 0 & 0 \end{pmatrix}
\end{equation*}
One may verify readily that $z_1B_1+ z_2 B_2 =b(z)-b(0)=b(z) =\Theta (z)$, and that these satisfy $B_1^*B_1 +B_2^*B_2 = I_4$ (and thus form a contractive solution to the Gleason problem, which is unique since $b = \Theta$ is quasi-extreme).

\bibliography{ACforDA}

\begin{thebibliography}{10}

\bibitem{Paulsen-rkhs}
V.~Paulsen and M.~Raghupathi.
\newblock {\em An Introduction to the theory of reproducing kernel {H}ilbert
  spaces}.
\newblock Cambridge Studies in Advanced Mathematics, 2016.

\bibitem{Hoff}
K.~Hoffman.
\newblock {\em Banach spaces of analytic functions}.
\newblock Courier Corporation, 2007.

\bibitem{dBss}
L.~deBranges and J.~Rovnyak.
\newblock {\em Square summable power series}.
\newblock Courier Corporation, 2015.

\bibitem{dBmodel}
L.~de~Branges and J.~Rovnyak.
\newblock Appendix on square summable power series, {C}anonical models in
  quantum scattering theory.
\newblock In {\em Perturbation theory and its applications in quantum
  mechanics}, pages 295--392. Wiley, 1966.

\bibitem{Nik1986}
N.K. Nikolskii and V.I. Vasyunin.
\newblock Notes on two function models in the {B}ieberbach conjecture.
\newblock In {\em The Bieberbach conjecture: Proceedings of the symposium on
  the occasion of the proof, Math. Surveys}, volume~21, pages 113--141. Amer.
  Math. Soc., 1986.

\bibitem{Ball1987}
J.A. Ball and T.L. Kriete.
\newblock Operator-valued {N}evanlinna-{P}ick kernels and the functional models
  for contraction operators.
\newblock {\em Integral Equations Operator Theory}, 10:17--61, 1987.

\bibitem{Ball2011dBR}
J.A. Ball and V.~Bolotnikov.
\newblock Canonical transfer-function realization for {S}chur multipliers on
  the {D}rury-{A}rveson space and models for commuting row contractions.
\newblock {\em Indiana Univ. Math. J.}, 61:665--716, 2011.

\bibitem{Sarason-dB}
D.~Sarason.
\newblock {\em Sub-{H}ardy {H}ilbert spaces in the unit disk}.
\newblock John Wiley \& Sons Inc., New York, NY, 1994.

\bibitem{Martin-dBmodel}
R.T.W. Martin.
\newblock Gleason solution models for non-commuting row contractions.
\newblock In preparation., 2016.

\bibitem{Clark1972}
D.N. Clark.
\newblock One dimensional perturbations of restricted shifts.
\newblock {\em J. Anal. Math.}, 25:169--191, 1972.

\bibitem{Jur2014AC}
M.T. Jury.
\newblock Clark theory in the {D}rury--{A}rveson space.
\newblock {\em J. Funct. Anal.}, 266:3855--3893, 2014.

\bibitem{Ross2006CT}
J.A. Cima, A.L. Matheson, and W.T. Ross.
\newblock {\em The {C}auchy transform}.
\newblock Number 125. Amer. Math. Soc., 2006.

\bibitem{Pop96disk}
G.~Popescu.
\newblock Non-commutative disc algebras and their representations.
\newblock {\em Proc. Amer. Math. Soc.}, 124:2137--2148, 1996.

\bibitem{Davidson2001}
K.R. Davidson, E.~Katsoulis, and D.R. Pitts.
\newblock The structure of free semigroup algebras.
\newblock {\em J. Reine Angew. Math.}, 533.

\bibitem{Sha2013}
O.M. Shalit.
\newblock Operator theory and function theory in {D}rury--{A}rveson space and
  its quotients.
\newblock In {\em Handbook of Operator Theory}, pages 1125--1180. Springer,
  2015.

\bibitem{McPutinar2005}
J.E. McCarthy and M.~Putinar.
\newblock Positivity aspects of the {F}antappie transform.
\newblock {\em J. Anal. Math.}, 97:57--82, 2005.

\bibitem{Jur2010Herglotz}
M.T. Jury.
\newblock Operator-valued {H}erglotz kernels and functions of positive real
  part on the ball.
\newblock {\em Compl. Anal. Oper. Theory}, 4:301--317, 2010.

\bibitem{BT1998DFP}
J.A. Ball and T.T. Trent.
\newblock Unitary colligations, reproducing kernel {H}ilbert spaces, and
  {N}evanlinna--{P}ick interpolation in several variables.
\newblock {\em J. Funct. Anal.}, 157:1--61, 1998.

\bibitem{NF}
B.~Sz.-Nagy and C.~Foia\c{s}.
\newblock {\em Harmonic analysis of operators on \uppercase{H}ilbert space}.
\newblock American Elsevier publishing company, Inc., New York, N.Y., 1970.

\bibitem{Pop89iso}
G.~Popescu.
\newblock Isometric dilations for infinite sequences of noncommuting operators.
\newblock {\em Trans. Amer. Math. Soc.}, 316:523--536, 1989.

\bibitem{DPac}
K.R. Davidson, J.~Li, and D.R. Pitts.
\newblock Absolutely continuous representations and a {K}aplansky density
  theorem for free semigroup algebras.
\newblock {\em J. Funct. Anal.}, 224:160--191, 2005.

\bibitem{Pop98universal}
G.~Popescu.
\newblock Universal operator algebras associated to contractive sequences of
  non-commuting operators.
\newblock {\em J. London Math. Soc. (2)}, 58:467--479, 1998.

\bibitem{Dav2011}
K.R. Davidson and E.G. Katsoulis.
\newblock Dilation theory, commutant lifting and semicrossed products.
\newblock {\em Doc. Math.}, 16:781--868, 2011.

\bibitem{Sar1965semi}
D.~Sarason.
\newblock On spectral sets having connected complement.
\newblock {\em Acta Sci. Math.(Szeged)}, 26:289--299, 1965.

\bibitem{Ball2003rkhs}
J.A. Ball and V.~Vinnikov.
\newblock Formal reproducing kernel {H}ilbert spaces: the commutative and
  noncommutative settings.
\newblock In {\em Reproducing kernel spaces and applications, Oper. Theory Adv.
  Appl.}, volume 143, pages 77--134.

\bibitem{Ball2006Fock}
J.A. Ball, V.~Bolotnikov, and Q.~Fang.
\newblock Schur-class multipliers on the {F}ock space: de {B}ranges-{R}ovnyak
  reproducing kernel spaces and transfer-function realizations.
\newblock In {\em Operator Theory, Structured Matrices, and Dilations, Theta
  Series Adv. Math., Tiberiu Constantinescu Memorial Volume}, volume~7, pages
  101--130. 2007.

\bibitem{ArvIII}
W.B. Arveson.
\newblock Subalgebras of {C}$^*$-algebras {I}{I}{I}.
\newblock {\em Acta Math.}, 181:159--228, 1998.

\bibitem{Ball2008}
J.A. Ball, V.~Bolotnikov, and Q.~Fang.
\newblock Schur-class multipliers on the {A}rveson space: de
  {B}ranges--{R}ovnyak reproducing kernel spaces and commutative
  transfer-function realizations.
\newblock {\em J. Math. Anal. Appl.}, 341:519--539, 2008.

\bibitem{Glea1964}
A.M. Gleason.
\newblock Finitely generated ideals in {B}anach algebras.
\newblock {\em J. Math. Mech.}, 13:125, 1964.

\bibitem{Alpay2002}
D.~Alpay and K.~H. Turgay.
\newblock Gleason's problem and homogeneous interpolation in {H}ardy and
  {D}irichlet-type spaces of the ball.
\newblock {\em J. Math. Anal. Appl.}, 276:654--672, 2002.

\bibitem{Ball2010}
J.A. Ball and V.~Bolotnikov.
\newblock Canonical de {B}ranges-{R}ovnyak model transfer-function realization
  for multivariable {S}chur-class functions.
\newblock In {\em CRM Proceedings and Lecture Notes}, volume~51, pages 1--40,
  2010.

\bibitem{McTrent2000}
S.~McCullough and T.T. Trent.
\newblock Invariant subspaces and {N}evanlinna--{P}ick kernels.
\newblock {\em J. Funct. Anal.}, 178:226--249, 2000.

\bibitem{Arv1998curv}
W.B. Arveson.
\newblock The curvature invariant of a {H}ilbert module over $\mathbb{C}
  [z_1,..., z_d]$.
\newblock {\em Proc. Natl. Acad. Sci. USA}, 96:11096--11099.

\bibitem{GRS2002}
D.C.V. Greene, S.~Richter, and C.~Sundberg.
\newblock The structure of inner multipliers on spaces with complete
  {N}evanlinna-{P}ick kernels.
\newblock {\em J. Funct. Anal.}, 194:311--331, 2002.

\end{thebibliography}

\end{document}